\documentclass{amsart}
\numberwithin{equation}{section}
\usepackage[all]{xy}
\usepackage{amssymb}
\usepackage{amsmath}
\usepackage{amsthm}
\let\cal\mathcal
\def\Ascr{{\cal A}}

\def\Cscr{{\cal C}}
\def\Dscr{{\cal D}}
\def\Escr{{\cal E}}
\def\Fscr{{\cal F}}

\def\Mscr{{\cal M}}

\def\Oscr{{\cal O}}
\def\Pscr{{\cal P}}

\def\Tscr{{\cal T}}

\def\Vscr{{\cal V}}

\let\blb\mathbb

\def \PP{{\blb P}}
\def \AA{{\blb A}}
\def \ZZ{{\blb Z}}

\def \NN{{\blb N}}
\def \RR{{\blb R}}
\def \HH{{\blb H}}

\def\udim{\operatorname{\underline{\dim}}}
\def\odim{\operatorname{\overline{\dim}}}

\def\id{\text{id}}
\def\Id{\operatorname{id}}

\def\Lotimes{\overset{L}{\otimes}}

\def\Mod{\operatorname{Mod}}
\def\mod{\operatorname{mod}}

\def\gr{\operatorname{gr}}

\def\Supp{\mathop{\text{\upshape Supp}}}

\def\Qch{\operatorname{Qch}}
\def\coh{\mathop{\text{\upshape{coh}}}}

\def\gr{\operatorname {gr}}
\def\Spec{\operatorname {Spec}}
\def\Rep{{{{\Mscr}}}}

\def\Ext{\operatorname {Ext}}
\def\Hom{\operatorname {Hom}}

\def\End{\operatorname {End}}
\def\RHom{\operatorname {RHom}}

\def\coker{\operatorname {coker}}
\def\ker{\operatorname {ker}}

\def\End{\operatorname {End}}

\def\id{{\operatorname {id}}}

\def\rk{\operatorname {rk}}

\def\r{\rightarrow}

\DeclareMathOperator{\Ind}{Ind}

\newtheorem{lemma}{Lemma}[section]
\newtheorem{proposition}[lemma]{Proposition}
\newtheorem{theorem}[lemma]{Theorem}
\newtheorem{corollary}[lemma]{Corollary}

\newtheorem{lemmas}{Lemma}[subsection]
\newtheorem{propositions}[lemmas]{Proposition}
\newtheorem{theorems}[lemmas]{Theorem}

\theoremstyle{definition}

\newtheorem{example}[lemma]{Example}

\newtheorem{examples}[lemmas]{Example}
\newtheorem{definitions}[lemmas]{Definition}

{

}

\theoremstyle{remark}

\newtheorem{remarks}[lemmas]{Remark}

\newdimen\uboxsep \uboxsep=1ex
\def\uboxn#1{\vtop to 0pt{\hrule height 0pt depth 0pt\vskip\uboxsep
\hbox to 0pt{\hss #1\hss}\vss}}

\def\uboxs#1{\vbox to 0pt{\vss\hbox to 0pt{\hss #1\hss}
\vskip\uboxsep\hrule height 0pt depth 0pt}}

\def\Az{\operatorname{Az}}

\def\HH{\mathsf{HH}}
\def\ess{\operatorname{ess}}
\def\Ob{\operatorname{Ob}}
\def\trdeg{\operatorname{trdeg}}
\def\Perf{\operatorname{Perf}}
\def\Qcoh{\operatorname{Qcoh}}
\def\Fp{\operatorname{Fp}}
\def\BB{\mathbb{B}}
\let\mathscr\mathcal
\title[Scalar extensions of derived categories]{Scalar extensions of derived categories and non-Fourier-Mukai functors}
\author{Alice Rizzardo}
\email[Alice Rizzardo]{arizzardo@sissa.it}
\address{SISSA\\ Via Bonomea, 265\\34136 Trieste\\Italy}
\author{Michel Van den Bergh}
\email[Michel Van den Bergh]{michel.vandenbergh@uhasselt.be}
\address{Universiteit Hasselt\\ Universitaire Campus\\ 3590 Diepenbeek\\Belgium}
\thanks{The second author is a senior researcher at the FWO}
\thanks{This research started at the Mathematical Sciences Research
Institute in 2013 with the support of the National Science Foundation under Grant
No.\ 0932078\,000. A number of results  were obtained during a research visit of the first author to the University of Hasselt under the
ESF Exchange Grant 4498  in the framework of the project ``Interactions of Low-Dimensional Topology and Geometry with Mathematical Physics (ITGP)''. The
authors are very grateful to the NSF, the MSRI, the ESF and the steering committee of the ITGP project for their support.}
\keywords{Fourier-Mukai functor, Orlov's theorem}
\subjclass{13D09, 18E30, 14A22}
\begin{document}
\begin{abstract}
Orlov's famous representability theorem asserts that any fully faithful exact functor between the bounded derived categories
of coherent sheaves on smooth projective varieties is a Fourier-Mukai functor. This result has been
extended by Lunts and Orlov to include functors from perfect complexes to quasi-coherent complexes.
In this paper we show that the latter extension is false without the full faithfulness hypothesis.

Our results are based on the properties of scalar extensions of derived categories, whose
investigation was started by Pawel Sosna and the first author.
\end{abstract}
\maketitle
\setcounter{tocdepth}{1}
\tableofcontents
\section{Introduction}
Unless otherwise specified, $k$ is an algebraically closed base field
of characteristic zero.  Orlov's famous representability
theorem~\cite[Thm 2.2]{Orlov4} asserts that any fully faithful exact functor
between the bounded derived categories of coherent sheaves on smooth
projective varieties over $k$ is a Fourier-Mukai functor. It is still unknown
if the full faithfulness hypothesis is necessary in this theorem,
although some positive results were obtained by the first author in
\cite{Rizzardo1}.

A number of extensions and variants of Orlov's theorem are known. See e.g.\ \cite{Ballard,COV,CS4,CS3,CS2,Kawamata2,OrlovLunts}. For an excellent survey on the current state of knowledge see \cite{CS1}.
In particular, Lunts and Orlov proved the following natural extension of Orlov's theorem to quasi-coherent sheaves:
{\def\thelemma{A}
\begin{proposition} \cite[Corollary 9.13(2)]{OrlovLunts}
\label{ref-A-0}
Let $X/k$ be a projective scheme such that~$\Oscr_X$ has no zero dimensional torsion
 and let $Y$ be
 a quasi-compact separated scheme. Then every fully faithful exact functor $\Psi :\Perf (X)\r
  D(\Qcoh(Y))$ is isomorphic to the restriction of a Fourier-Mukai functor associated to an
object in $D(\Qcoh(X\times Y))$.
\end{proposition}
} 
One of the main results of this paper is that this extension
is false if we drop the condition that $\Psi$ is fully
faithful, even in the case that $X$, $Y$ are smooth and projective (see Theorem \ref{ref-8.1-43} below).
Our arguments are based on the properties of scalar extensions of
derived categories, which we will outline below.  We will get back to
Proposition \ref{ref-A-0} at the end of the introduction.

\medskip

 If $\mathfrak{a}$ is
a $k$-linear category and $B$ is a $k$-algebra, we denote by $\mathfrak{a}_B$ the category of $B$-objects
in $\mathfrak{a}$, i.e.\ pairs $(M,\rho)$ where $M\in \Ob(\mathfrak{a})$ and $\rho:B\r \mathfrak{a}(M,M)$
is a~$k$-algebra morphism. 
If $\Cscr$ is abelian then so is $\Cscr_B$, but if $\Tscr$ is triangulated
there is no reason for this to be the case for $\Tscr_B$ as well.

While investigating generalizations of Orlov's theorem \cite{Rizzardo}
the first
author studied the obvious forgetful functor
\[
F:D^b(\Cscr_B)\r D^b(\Cscr)_B
\]
for $B/k=L/k$ a field extension. She proved an essential surjectivity result for $\trdeg L/k\le 2$ but it appeared difficult
to go beyond that. Indeed, in the present paper we will show that $F$ is generally \emph{not} essentially surjective when $\trdeg L/k=3$.
To put this in context, we start with a positive result which is naturally proved using $A_\infty$-techniques:
{
\def\thelemma{B}
\begin{proposition} (See Propositions \ref{ref-9.1.1-48},\ref{ref-9.1.2-49},\ref{ref-9.1.3-50} below.)  Assume that $\Cscr$ is a Grothendieck category. 
\label{ref-B-1}
\begin{itemize}
\item If $B/k$ has Hochschild dimension $\le 2$, $F$ is essentially surjective.
\item If $B/k$ has Hochschild dimension $\le 1$, $F$ is in addition full.
\item If $B/k$ has Hochschild dimension $0$, $F$ is an equivalence
of categories. 
\end{itemize}
\end{proposition}
}
Recall that, for a finitely generated field extension $L/k$, the Hochschild dimension is equal to the transcendence degree.
Proposition \ref{ref-B-1} represents a strengthening of the results in \cite{Keller21}.
The case of Hochschild dimension $0$ generalises results by Sosna \cite{Sosna}.

However, our next result shows that one cannot hope to substantially improve Proposition \ref{ref-B-1}:
{\def\thelemma{C}
\begin{theorem} (See Theorem \ref{ref-7.1-41} below.) 
\label{ref-C-2}
Let $X/k$ be a smooth connected projective variety which is not a point, a projective line or an elliptic curve. Then there exists a finitely generated field extension $L/k$
  of transcendence degree $3$ together with an object $Z\in D^b(\Qcoh(X))_L$ which is not
in the essential image of $F$.
\end{theorem}
}
The proof
of this theorem will depend on a similar  result for representations of wild quivers, which we will prove first (see Proposition \ref{ref-6.1-34} below). 

\medskip

As the reader may notice, Theorem \ref{ref-C-2} leaves out the case where $X$ is
a curve of genus $\le 1$.  
The key point is that in this case the 
moduli space of indecomposable objects has dimension  $\le 1$. We capture this in the concept of ``essential dimension'', which is
roughly speaking the minimal number of parameters required to define any family of indecomposable objects (see
Definition \ref{ref-5.2.1-30} below for a precise definition). 
From this theory it follows that if $X$ is a curve of genus $\le 1$ and $\Cscr=\Qcoh(X)$, for any field extension $L/k$ the essential image 
of 
$F$ contains all objects in $D^b(\Qcoh(X))_L$ whose cohomology lies in $\coh(X_L)\subset \Qcoh(X_L)\cong \Qcoh(X)_L$
(see Remark \ref{ref-5.2.2-31} and Theorem \ref{ref-5.2.3-32} below).
 We suspect that $F$ is in fact essentially surjective, but we have not proved it. 

\medskip

Now we come back to Proposition \ref{ref-A-0}. A counterexample to this proposition,
when dropping the full faithfulness hypothesis, may be obtained using the following result:
{\def\thelemma{D}
\begin{theorem} (See Theorem \ref{ref-8.1-43} below)
Let $X$, $Y$ be connected smooth
projective schemes. Let $i_\eta:\eta\r X$ be the generic point of $X$ and let
$L=k(\eta)$ be the function field of~$X$. Assume that
$D^b(\Qcoh(Y))_L$ contains an object $Z$ which is not in the essential
image of $D^b(\Qcoh(Y)_L)$ (for example as in Theorem \ref{ref-C-2}). Define  $\Psi$ as the 
composition
\[
\xymatrix{
& \Perf(X) \ar[r]^-{i^*_{\eta}} &  D(L) \ar[r]^-{L\mapsto Z} & D(\Qcoh(Y))
}.
\]
Then $\Psi$ is not the restriction of a Fourier-Mukai functor.
\end{theorem}
}

\bigskip

We start with a few technical sections that will provide tools for the proofs of the main results. The
impatient reader may wish to  proceed directly to \S\ref{ref-5-26} at first reading.
\section{Acknowledgements}
The authors thank Greg Stevenson and Adam-Christiaan Van Roosmalen for
a number of interesting discussions which reinforced our interest in the
problem.

The first author would like to thank the University of Hasselt for its hospitality during her visit to Belgium. She would also like to thank MSRI for providing a welcoming environment during the Noncommutative Geometry and Representation Theory semester, which encouraged many fruitful conversations, some of which got this project started. Finally, she would like to thank her advisor, Johan de Jong, for suggesting she look into scalar extensions of derived categories.

\section{Moduli spaces of representations of algebras}
Moduli space of representations for algebras 
may be constructed in several different ways \cite{King,Procesi2}.
We remind the reader of a construction which is based on the properties
of the Formanek center and which will be used in the proof Lemma \ref{ref-2.2.1-3} 
and Proposition \ref{ref-2.5.2-11}, which are the main results
of this section. We will use our standing characteristic zero hypothesis to simplify the discussion.
\subsection{The representation functor}
For a reference on this subject, the reader can consult \cite{Procesi2}. An Azumaya algebra is a matrix algebra for the \'etale topology. Let $A$ be a $k$-algebra and let $n>0$.  Consider the functor $\Az_{n,A}$
from commutative $k$-algebras to the category of sets defined as follows: if $R$
is a commutative $k$-algebra, $\Az_{n,A}(R)$ is the set of equivalence classes
of maps of $k$-algebras $\rho:A\r B$, where $B$ is an Azumaya algebra of rank $n^2$
over $R$ satisfying $\rho(A)R=B$. Two  maps $\rho:A\r B$,
$\rho':A\r B'$ are considered equivalent if there exists an isomorphism
of $R$-algebras $\xi:B\r B'$ such that $\xi\rho=\xi'$. The functor $\Az_{n,A}$ is
a sheaf for the Zariski topology, and hence extends canonically to a functor
from $k$-schemes to the category of sets.

The functor $\Az_{n,A}$ is representable in the category of $k$-schemes
(see \cite[Ch IV, Thm 1.8 and Ch VIII, Thm 2.2]{Procesi2}). The representing scheme may be constructed as a  ``Formanek
center'', constructed as follows. Let $\Lambda$ be a $k$-algebra. The identities of $n\times n$ matrices with coefficients in $\Lambda$ are defined as
$$I=\{f(x_s)\in \Lambda\{X_s\} | f(r_s)=0 \text{ for all }r_s \in (R)_n\},$$
where $(R)_n$ denotes the algebra of $n\times n$ matrices with coefficients in $R$, and $R$ is a generic commutative algebra in the variety of commutative algebras (see \cite[Definition 4.3]{Procesi2}). If a ring $C$ satisfies the identities of $n\times n$ matrices, the Formanek center $F(C)$ of $C$ is defined as the subring of $C$ obtained by evaluating all central polynomials of $n\times n$-matrices without constant term. A central polynomial for $n\times n$ matrices is a polynomial in non-commuting variables that is non-constant, but yields a scalar matrix whenever it is evaluated at $n\times n$ matrices.

By definition, $F(C)\subset Z(C)$. Since the field is of characteristic zero, the usual polarization argument \cite{Procesi3} shows that we may compute the Formanek center by evaluating central polynomials which are homogeneous of degree one in every variable. From this it easily follows that it is an ideal.\footnote{We do not know if this is true in finite characteristic.} An algebra satisfying the identities 
of $n\times n$-matrices is Azumaya of rank $n^2$ over its center if and only if the Formanek center is equal to the ordinary center
(this follows easily from \cite[Ch VIII, Th 2.1(6)]{Procesi2}, using the characteristic zero hypothesis again).

Let ${A}_n$
be the quotient of $A$ by the identities of $n\times n$-matrices, and let $F_{n,A}=F(A_n)$ be the Formanek center
of $A_n$. Put $F^e_{n,A}=F^e(A_n)\overset{\text{def}}{=}k+F(A_n)\subset A_n$ and
$\tilde{U}_{n,A}=\Spec F^e_{n,A}$. Let $\tilde{\Ascr}_n$ be the 
sheaf of algebras on $\tilde{U}_{n,A}$ associated to~${A}_n$ and put $U_{n,A}\overset{\text{def}}{=}
\tilde{U}_{n,A}-V(F_{n,A})$. Finally, let $\Ascr_n$ be the restriction of $\tilde{\Ascr}_n$
to $U_{n,A}$. Then $\Ascr_n$ is a sheaf of Azumaya algebra of rank $n^2$ with center equal to $\Oscr_{U_{n,A}}$ (presumably this depends 
on the characteristic zero hypothesis, see \cite[Ch VIII, Cor.\ 2.3]{Procesi2} for a result valid in any characteristic). Note that $U_{n,A}$ has an affine covering by
schemes of the form
$U_{n,A,f}=\Spec (F^e_{n,A})_f$, where
$f$ runs through the elements of $F_{n,A}$. The global sections of $\Ascr_n$
restricted to $U_{n,A,f}$ are given by $(A_{n})_f$.

\medskip

It follows from \cite[Ch VIII, Thm 2.2]{Procesi2}   that
 $U_{n,A}$  is isomorphic to a differently constructed scheme
which
represents $\Az_{n,A}$.   For further reference we give a description of the bijection
between
$U_{n,A}(R)$ and
$\Az_{n,A}(R)$ as given in the proof of \cite[Ch VIII, Thm 2.2]{Procesi2}. Assume
$\rho:A\r B$ 
represents an element of $\Az_{n,A}(R)$. The map $\rho$ descends to a map
$\rho_n:A_n\r B$, and hence to a map $\rho_{n,f}:(A_{n})_f\r B_{\rho_n(f)}$ for $f\in F_{n,A}$. We obtain
an induced map  $\rho_{n,f}:F^e((A_{n})_f)\r F^e(B_{\rho_n(f)})$.
Now $B_{\rho_n(f)}$ is an Azumaya algebra and hence $F^e(B_{\rho_n(f)})$ is equal to its center 
$R_{\rho_n(f)}$. It is easy to see that
the maps   $(F^e_{n,A})_f\r R_{\rho_n(f)}$ may be glued to a scheme map
$\rho_n:\Spec R\r \Spec F^e_{n,A}-V(F_{n,A})=U_{n,A}$.

Conversely, if we start from a scheme map $\rho_n:\Spec R\r U_{n,A}$, then we put
$B=\rho_n^\ast(\Ascr_n)$ (where here and below we usually identify quasi-coherent sheaves on affine schemes with 
their global sections).  Since the elements of $A$ restrict to sections
of $\Ascr_n$, we obtain a corresponding map $\rho:A\r B$. The required condition
$\rho(A)R=B$ is easily checked.

\subsection{The main lemma}
Let the notation be as in the previous section. 
For a not necessarily closed point $i_x:x\r U_{n,A}$  
we say that $x$ is split 
if $i^\ast_x(\Ascr_{n})$ is split as a central simple algebra, i.e. if it  is isomorphic to $M_n(k(x))$ as a $k(x)$-algebra.
In this case, $V_x$ is defined to be the corresponding irreducible $A_{k(x)}\overset{\text{def}}{=}k(x)\otimes_k A$-representation
of dimension $n$ over $k(x)$.
\begin{lemmas} \label{ref-2.2.1-3} Assume $x$ is split. We have
$
\End_A(V_x)=k(x)
$.
\end{lemmas}
The point of the lemma is that the endomorphisms are only assumed to be $A$-linear, not $A_{k(x)}$-linear.
\begin{proof} Let $O(x)$ be the image of $F^e_{n,A}$ in $k(x)$. Then $k(x)$ is the field
of fractions of $O(x)$ and we have 
\begin{align*}
\End_A(V_x)&=\End_{{A}_n}(V_x)\\
&=\End_{O(x)\otimes_{F^e_{n,A}}{A}_n}(V_x)\\
&=\End_{k(x)\otimes_{F^e_{n,A}}{A}_n}(V_x)\\
&=\End_{A_{k(x)}}(V_x)\\
&=k(x).
\end{align*}
In the second equality we use that ${A}_n\r O(x)\otimes_{F^e_{n,A}}{A}_n$ is surjective.
In the third equality we use that $O(x)\otimes_{F^e_{n,A}}{A}_n\r k(x)\otimes_{F^e_{n,A}}{A}_n$
is an epimorphism of rings and the fact that $V_x$ is a $k(x)\otimes_{F^e_{n,A}}{A}_n$-module. For the
fourth equality we use that $A_{k(x)}\r k(x)\otimes_{F^e_{n,A}}{A}_n$ is surjective.
\end{proof}
\begin{examples} \label{ref-2.2.2-4} Here is an example where one can check the conclusion of Lemma \ref{ref-2.2.1-3} directly.
Let $Q$ be the quiver with one vertex and three loops and $A=kQ=k\langle X,Y,Z\rangle$. Then it is easy to see
that $U_{1,A}\cong \AA^3$. 
Let $V_\eta$ be the representation corresponding to the generic point $\eta$ of $\AA^3$. It is 
defined over the field $L=k(\eta)=k(x,y,z)$ and has the form 
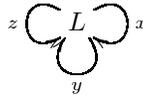
\[
\xymatrix{
&L \ar@(ur,dr)^x\ar@(dr,dl)^y\ar@(ul,dl)_z&\\
&&
}
\]
One easily checks that $\End_{k\langle X,Y,Z\rangle}(V_\eta)=L$.
\end{examples}
\subsection{The split representation functor}
In order to use Lemma \ref{ref-2.2.1-3}, we must be able to show that $i_x^\ast(\Ascr_n)$
is split. We discuss this next. For a $k$-scheme $X$, let  $\Rep_{n,A}(X)$ be the collection
of equivalence classes of quasi-coherent sheaves of left $A\otimes_k \Oscr_X$-modules $V$ on $X$ which are vector bundles or rank $n$
over $X$ such that for every point $x\in X$  we have that $i^\ast_x(V)$ is simple. We consider $V$ and $W$ to be equivalent if there exists an
invertible $\Oscr_X$-module $I$ such that $W=V\otimes_X I$. 
\begin{lemmas}
\label{ref-2.3.1-5}
 Assume that $\Rep_{n,A}$  is representable in the category of
$k$-schemes. Then $\Ascr_n$ is split and $\Rep_{n,A}$ is represented by $U_{n,A}$.
\end{lemmas}
\begin{proof} Let $M_{n,A}$ be the representing scheme for $\Rep_{n,A}$, and let $\Vscr_{n,A}$ be the universal bundle
on $M_{n,A}$ (determined up to tensoring with a line bundle).
We have
a natural transformation 
\[
\phi:\Rep_{n,A}\r \Az_{n,A}
\]
sending $V$ to $\End_X(V)$. This yields a morphism between the representing
schemes
\begin{equation}
\label{ref-2.1-6}
\phi:M_{n,A}\r U_{n,A}
\end{equation}
such that 
\begin{equation}
\label{ref-2.2-7}
\phi^\ast\Ascr_n=\End_{M_{n,A}}(\Vscr_{n,A}).
\end{equation}
Clearly, $\phi(X)$ is injective for any $X$, and surjective up to �tale coverings.
This means that $\phi$ is actually an isomorphism and hence $\Ascr_n$ is split
by \eqref{ref-2.2-7}.
\end{proof}
\subsection{Partitioning by ranks}
\label{ref-2.4-8}
 Let
$l=ke_1+\cdots+ke_m$ be the semi-simple $k$-algebra determined by
$e_ie_j=\delta_{ij}e_i$, $\sum_i e_i=1$. Let $A$ be an $l$-algebra.
By this we mean that there is given a $k$-algebra morphism $l\r
A$.  We denote the images of the $(e_i)_i$ in $A$ also by $(e_i)_i$. Fix
 positive natural numbers $\alpha=(\alpha_1,\ldots,\alpha_m)$
such that $|\alpha|\overset{\text{def}}{=}\sum_i\alpha_i=n$.
We let $\Az_{\alpha,A}(R)$ be the subset of $\Az_{n,A}(R)$ consisting of equivalence classes
$\rho:A\r B$ such that $\rk_{R} \rho(e_i)B\rho(e_i)=\alpha_i^2$ for all $i$.  It
is easy to see that $\Az_{\alpha,A}$ is an open subfunctor of
$\Az_{n,A}$, and furthermore $\Az_{n,A}=\coprod_{\alpha} \Az_{\alpha,A}$.  We
obtain a corresponding decomposition $U_{n,A}=\coprod_{\alpha}
U_{\alpha,A}$ for the representing spaces. The restriction of
$\Ascr_n$ to $U_{\alpha,A}$ will be denoted by $\Ascr_\alpha$.

In a similar way we may define functors $\Rep_{\alpha,A}\subset \Rep_{n,A}$, where we now
require that for $V\in \Rep_{\alpha,A}(X)$ the rank of $e_i V$ is equal to $\alpha_i$. We will say that 
 $V$ has \emph{dimension vector} $\alpha$.
We have the following obvious generalization of Lemma \ref{ref-2.3.1-5}:
\begin{lemmas}
\label{ref-2.4.1-9}
 Assume that $\Rep_{\alpha,A}$  is representable in the category of
$k$-schemes. Then $\Ascr_\alpha$ is split and $\Rep_{\alpha,A}$ is represented by $U_{\alpha,A}$.
\end{lemmas}
\subsection{Stability conditions}
\label{ref-2.5-10}
For $\lambda,\mu\in \ZZ^{m}$ write $\lambda\cdot\mu=\sum_{i=1}^m \lambda_i\mu_i$. 
Let\footnote{We assume $0\in\NN$.} $\alpha\in \NN^{m}$,
and choose $\lambda\in \ZZ^{m}$ such that $\lambda\cdot \alpha=0$.

If $K/k$ is a
field extension, then $V\in \Rep_{\alpha,A}(K)$ is $\lambda$-(semi-)stable \cite{King}  if for any proper $A_K$-subrepresentation
$0\neq W\subsetneq V$ with dimension vector $\beta$  we have
\[
\lambda\cdot \beta(\ge)> \lambda\cdot\alpha.
\]
We say that $V\in \Rep_{\alpha,A}(X)$ is $\lambda$-(semi-)stable if for any $i:\Spec K\r X$
with $K/k$ a field extension it is true that $i^\ast(V)$ is $\lambda$-(semi-)stable.
We denote the corresponding subfunctor of $\Rep_{\alpha,A}$  by $\Rep_{\alpha,\lambda,A}$.
Recall 
\begin{theorems}
Assume that $A$ is finitely generated over $k$ and that $\alpha$ is indivisible. Then
$\Mscr_{\alpha,\lambda,A}$ is
representable by a scheme of finite type over $k$.
\end{theorems}
This is  \cite[Prop 5.3]{King}. The given reference assumes $A$ to be finite dimensional, but the proof carries over completely in the case where $A$ is finitely generated over $k$.

We denote the representing
scheme by $M_{\alpha,\lambda,A}$, and the corresponding universal
$A$-representation by $\Vscr_{\alpha,\lambda,A}$.
We will prove the following generalization of Lemma \ref{ref-2.2.1-3}. The point is again that we take endomorphisms over $A$, and not over $A_{k(x)}$.  
\begin{propositions} \label{ref-2.5.2-11} Let $A$ be finitely generated, and let $\alpha$ be indivisible. 
For $i_x:x\r M_{\alpha,\lambda,A}$ put $V_x=i_x^\ast(\Vscr_{\alpha,\lambda,A})$.
Then 
\begin{equation}
\label{ref-2.3-12}
\End_{A}(V_x)=k(x).
\end{equation}
\end{propositions}
The proof of Proposition \ref{ref-2.5.2-11} uses the representation theory of quivers. See \cite{Brion} for a reference. 
Let $Q=(Q_0,Q_1,h,t)$ be a finite  quiver with vertices $Q_0$ and arrows $Q_1$. The maps
$h,t:Q_1\r Q_0$ associate an arrow with its head and tail. 
If $R$ is a $k$-algebra, we let $RQ$ be the path algebra of~$Q$ with coefficients in $R$. 
Any finitely generated $l$-algebra
in the sense of \S\ref{ref-2.4-8} is a quotient of a suitable path algebra $kQ$
with $l$ corresponding to $\bigoplus_{i\in Q_0} ke_i$, where $e_i$ is the length zero path in $Q$ associated to
the  vertex $i$. It is easy to see that it is sufficient to prove Proposition \ref{ref-2.5.2-11}
in the case $A=kQ$. Therefore we specialize to that case, and 
we will replace $A$ in the notation by $kQ$.

For a $kQ$-representation $V$ we write $V_i=e_i V$. Let $P_i=kQe_i$ be the
standard projective representation corresponding to $i\in Q_0$.
Recall the following:
\begin{propositions} \label{ref-2.5.3-13} (Green) Every projective $kQ$-module $P$ is of
the form $\bigoplus_{i\in Q_0} W_i\otimes_k P_i$, where $W_i$ is a $k$-vector space. The $W_i$ are 
uniquely determined by $P$.
\end{propositions}
\begin{proof} The fact that all projectives are of the indicated form is \cite[Cor.\ 5.5]{Green1}.
Then $W_i$ can be recovered from $P$ as $W_i=S_i\otimes_{kQ} P$, where $S_i$ is the standard
simple corresponding to vertex $i$.
\end{proof}
Below we will need a definition which is dual to the concept of dimension vector $\udim V$ of $V$. Assume that
$V\in \Mod(kQ)$ is finitely presented. Then $V$ has a projective resolution
\[
0\r \oplus_i P_i^{\oplus b_i} \r \oplus_i P_i^{\oplus a_i}\r V\r0
\]
with $P_i$ the projective $kQ$-representation associated to vertex $i$. We put
\[
(\odim V)_i=a_i-b_i.
\]
It follows from Proposition \ref{ref-2.5.3-13} that this is a well defined element of
$\ZZ^{Q_0}$.

For $W$ a finitely presented and $V$ a finite dimensional $kQ$-representation, one has
\[
 \odim W\cdot \udim V =\dim \Hom_{kQ}(W,V)-\dim \Ext_{kQ}^1(W,V).
\]
If $\Hom_{kQ}(W,V)=\Ext^1_{kQ}(W,V)=0$ then we write $W\perp V$.  Let $W^\perp$ be the category of all $V$ such that $W\perp V$. 
Recall the following result
\begin{theorems} \cite{WeymanDerksen,DomokosZubkov}\cite[Cor 1.1]{SchVdB3} \label{ref-2.5.4-14}
Assume that $V\in \Mod(kQ)$ is finite-dimensional. Then $V$ is $\lambda$-semi-stable
if and only there is a finitely presented   $0\neq W\in \Mod(kQ)$ such that 
$W\perp V$, and such that $\odim W$ is a strictly negative multiple of~$\lambda$.  
\end{theorems}
For $\lambda\cdot \alpha=0$, and for $W$ such that $\odim W=-n\lambda$, $n>0$ 
we define a subfunctor $\Mscr_{\alpha,W,Q}$
of $\Mscr_{\alpha,\lambda,Q}$ consisting of those representations $V$ with $\udim
V=\alpha$ such that $W\perp V$. 
If $\alpha$ is indivisible, then this subfunctor is representable  by an open subset $M_{\alpha,W,Q}$
of $M_{\alpha,\lambda,Q}$. 
The above theorem may be rephrased as
saying that $(M_{\alpha,W,Q})_W$ is an open covering of $M_{\alpha,\lambda,Q}$. We let $\Vscr_{\alpha,W,Q}$ be the restriction of
$\Vscr_{\alpha,\lambda,Q}$ to $M_{\alpha,W,Q}$.

Let 
\[
0\r P\xrightarrow{\delta} Q\r W\r 0
\]
be a minimal projective resolution of $W$ and let $(kQ)_\delta$ 
be the corresponding universal localisation.\footnote{If $A$ is a ring and $\delta:P\r Q$ is
map between finitely generated projective left $A$-modules then $A\r A_\delta$ is 
universal for the ring extensions $A\r B$ such that $B\otimes_A\delta$
is an isomorphism. $A\r A_\delta$ is an epimorphism in the category of
rings and a left $A$-module $M$ has a (necessarily unique) $A_\delta$-action
provided the functor $\Hom_A(-,M)$ transforms $\delta$ into an invertible morphism.}

Then $V\in W^\perp$ if and
only if $\Hom_{kQ}(\delta,V)$ is invertible. In other words, $V\in W^\perp$ if and only if the $kQ$-action on $V$ 
extends to a $(kQ)_{\delta}$-action. Thus we obtain
\begin{equation}
\label{ref-2.4-15}
W^\perp\cong \Mod((kQ)_\delta).
\end{equation}
Note further
\begin{lemmas} \cite[Lemma 3.1]{ALeB} Let $W$ be a finitely presented $kQ$-representation and let $V\in W^\perp\cap \Mod(KQ)$ be finite dimensional over $K$, where $K/k$ is a field extension. 
Then $V$ is simple in $W^\perp\cap \Mod(KQ)$ if and only if $V$ is $-\odim W$-stable.
\end{lemmas}
\begin{proof} The blanket hypothesys in \cite{ALeB} that $Q$ should not have any oriented cycles is not necessary for this particular lemma. For the benefit of the reader we recall the proof of the implication
simple$\Rightarrow$stable, as it is instructive. Assume that $V$ is simple in $W^\perp\cap \Mod(KQ)$,
and let $0\neq V'\subsetneq V$ be a $KQ$-subrepresentation. Since $\Hom_{kQ}(W,V')=0$, we have
$\odim W\cdot \udim V'\le 0$. If $\odim W\cdot \udim V'=0$ then we also have $\Ext^1_{kQ}(W,V')=0$
and hence $W\perp V'$. Since $V$ is simple this is a contradiction. It follows
$\dim \Ext^1_{kQ}(W,V')>0$, and so $\odim W\cdot \udim V'<0$.
\end{proof}
From this we easily obtain an isomorphism of functors
\[
\Rep_{\alpha,(kQ)_\delta}\cong \Mscr_{\alpha,W,Q}
\]
and hence, by Lemma \ref{ref-2.4.1-9}, it follows that $\Ascr_{\alpha,(kQ)_\delta}$ is split and that there is an isomorphism
\[
U_{\alpha,(kQ)_\delta}\cong M_{\alpha,W,Q}.
\]
\begin{proof}[Proof of Proposition \ref{ref-2.5.2-11}] There exists some $W$ such that
$x\in M_{\alpha,W,Q}$. Now let the notation be as above. Then by \eqref{ref-2.4-15}
\[
\End_{kQ}(V_x)=\End_{(kQ)_\delta}(V_x).
\]
It now suffices to invoke Lemma \ref{ref-2.2.1-3}.
\end{proof}
\begin{examples}
\label{ref-2.5.6-16}
Let $Q$ be the generalised Kronecker quiver with 4 arrows.
\[
\xymatrix{
*+[o][F-]{\scriptscriptstyle 1} \ar@<3ex>[r]|{T} \ar@<1ex>[r]|{X} \ar@<-1ex>[r]|{Y}  \ar@<-3ex>[r]|{Z} &*+[o][F-]{\scriptscriptstyle 2}
}.
\]
Let $\alpha=(1,1)$ and $\lambda=(-1,1)$. A representation of dimension vector $\alpha$ is $\lambda$-stable if not all arrows are zero, and two such representations are isomorphic if one is obtained from the other by multiplying all arrows by the same scalar. This corresponds to a point in $\PP^3$. From there one easily shows that $M_{\alpha,\lambda,A}=\PP^3$. Let $V_\eta$ be the representation corresponding to the generic point $\eta$ of $\PP^3$. It is 
defined over the field $L=k(\eta)=k(x,y,z)$ and has the form 
\[
\xymatrix{
L \ar@<3ex>[r]|{1} \ar@<1ex>[r]|{x} \ar@<-1ex>[r]|{y}  \ar@<-3ex>[r]|z & L
}.
\]
One easily checks that $\End_{A}(V_\eta)=L$.

If $P_{1}$, $P_2$ denote the projective $Q$-representations corresponding to the vertices $1$, $2$ and $W=\coker(P_2\xrightarrow{T} P_1)$ then $V_\eta\in W^\perp$.
The universal localisation of $kQ$ at $T$ is obtained by adjoining to $Q$ an inverse arrow $T^{-1}$ from $2$ to $1$. One checks that $(kQ)_T$ is Morita equivalent to 
$k\langle X,Y,Z\rangle$, and this example reduces in fact to Example \ref{ref-2.2.2-4}.
\end{examples}

It is difficult to say when $M_{\alpha,\lambda,Q}$ is non-empty, but Proposition \ref{ref-2.5.7-17} below will be
sufficient for our purposes. For $\alpha,\beta\in \ZZ^{Q_0}$ write
\[
\langle \alpha,\beta\rangle=\sum_i\alpha_i\beta_i-\sum_{a\in Q_1} \alpha_{t(a)}\beta_{h(a)}.
\]
If $W$, $V$ are finite dimensional $kQ$-representations then
\[
 \langle \udim W, \udim V \rangle=\dim \Hom_{kQ}(W,V)-\dim \Ext_{kQ}^1(W,V).
\]
Put 
\[
(\alpha,\beta)=\langle \alpha,\beta\rangle+\langle \beta,\alpha\rangle
\]
and let $e_i\in \NN^{Q_0}$ be such that $(e_i)_j=\delta_{ij}$.  The fundamental region \cite{Kac1} is defined as 
\[
F(Q)=\{\alpha\in \NN^{Q_0}\mid \forall i: (e_i,\alpha)\le 0, \text{$\alpha$ has connected support}\}.
\]
Recall
\begin{lemmas}
$F(Q)$ is empty for a Dynkin quiver and is spanned by a single vector
$\delta$ satisfying $(\delta,\delta)=0$ in the case that $Q$ is
extended Dynkin. If $Q$ contains a component with more than one vertex which is not Dynkin
or extended Dynkin then $F(Q)$ contains indivisible
$\alpha$ such that $(\alpha,\alpha)$ is an arbitrarily large negative
number. 
\end{lemmas}
\begin{proof}
The first two cases are well known. So supposed that $Q$ is connected and not Dynkin or extended Dynkin and $|Q_0|>1$
By
\cite[Lemma 1.2]{Kac1}  there exists $\alpha\in \NN^{Q_0}$ such that
all $\alpha_i>0$ and  $(e_i,\alpha)<0$ for each $i$. In other
words the cone
\[
C(Q)=\{\beta\in \RR^{Q_0}\mid \forall i: \beta_i\ge 0, (e_i,\beta)\le 0\}.
\]
had non empty interior and is of dimension
$|Q_0|>1$. So $\operatorname{int} C(Q)\cap \ZZ^{Q_0}$
 contains infinitely many indivisible elements (for example take the minimal elements in smaller and smaller subcones which are disjunct except for 0). 

Now if $\beta\in \operatorname{int} C(Q)\cap \ZZ^{Q_0}$ then $(\beta,\beta)=\sum_i\beta_i (\beta,e_i)<-\sum_i\beta_i$. So for any $N>0$ the set
\[
\{\beta\in \operatorname{int} C(Q)\cap \ZZ^{Q_0}\mid (\beta,\beta)>-N\}
\]
is finite. This shows that $(\beta,\beta)\r -\infty$.
\end{proof}
\begin{remarks} It follows that if $Q$ is a connected ``wild'' quiver
  (i.e.\ not Dynkin or extended Dynkin) which is not the two loop quiver
  then $F(Q)$ contains an indivisible vector $\alpha$ such that
  $(\alpha,\alpha)\le -4$. This fact will be used below.
\end{remarks}
\begin{propositions} \label{ref-2.5.7-17} \cite{Kac1,King,Schofield99} Let $\alpha$ be an indivisible dimension vector in $F(Q)$. Then there exists some
$\lambda$ satisfying $\lambda\cdot \alpha=0$ such
that $M_{\alpha,\lambda,Q}$ is non-empty. In that case $M_{\alpha,\lambda,Q}$  has dimension $-(1/2)(\alpha,\alpha)+1$.
\end{propositions}
\begin{proof}
Since $\alpha\in F(Q)$, the generic $Q$-representation  with dimension vector $\alpha$ is a Schur representation (i.e.\ it has only trivial endomorphisms) \cite{Kac1}. Therefore it is stable for suitable $\lambda$ by
\cite[Theorem 6.1]{Schofield92}. The dimension maybe computed using the standard fact that $\dim M_{\alpha,Q}=\dim \Ext^1_{kQ}(V,V)=-\langle \udim V,\udim V\rangle+1$
for $V$ generic. 
\end{proof}

\section{Moduli spaces of vector bundles on curves}
\label{ref-3-18}
In this section we prove an analogue of Proposition \ref{ref-2.5.2-11} for vector bundles on curves.
Below $X$ is a smooth projective curve over $k$ of genus $g$.
The theory of moduli spaces of vector bundles on curves is well known, so we will not repeat it here (see e.g.\ \cite{newstead}).

Given $r,d$ such that $\gcd(r,d)=1$, 
the functor $\Mscr_{r,d}$ of families of stable vector bundles of rank $r$
and degree $d$ on $X$ has a fine moduli space $M_{r,d}$  such that
\[
\dim M_{r,d}=1+r^2(g-1).
\]
Let $\Vscr_{r,d}$ be the universal bundle on $M_{r,d}$. We will prove the following analogue of 
Proposition \ref{ref-2.5.2-11}
\begin{proposition} \label{ref-3.1-19} Let $x\in M_{r,d}$ and put $V_x=i_x^\ast(\Vscr_{r,d})$. Let $p:X_{k(x)}\r X$
be the map obtained by base extension from the structure map $\Spec k(x)\r \Spec k$. Then
\[
\End_X(p_\ast V_x)=k(x).
\]
\end{proposition}
To prove Proposition \ref{ref-3.1-19} 
we will use the
following analogue of Theorem \ref{ref-2.5.4-14},
which is  a fundamental result by Faltings:
\begin{theorem} \cite{Faltings}
Let $X$ be a smooth projective curve. A bundle $\Escr$ on $X$ is semi-stable if there exists a non-zero
bundle $\Fscr$ such $\Fscr\perp\Escr$.
\end{theorem}
As before $\Fscr\perp\Escr$ if $\Hom_X(\Fscr,\Escr)=0$, $\Ext^1_X(\Fscr,\Escr)=0$. 
Given $\Fscr\in \coh(X)$, we define as in the quiver
case a subfunctor $\Mscr_{r,d,\Fscr}$ of $\Mscr_{r,d}$ consisting of those families in $\Mscr_{r,d}$
that are right orthogonal to $\Fscr$. This subfunctor is representable by
an open subset of $M_{r,d}$, which we denote by $M_{r,d,\Fscr}$.

Let $\Fscr\in \coh(X)$ be such that $\operatorname{\Supp} \Fscr=X$.  Put
\[
\Fscr^\perp=\{\Escr\in \Qcoh(X)\mid
\Hom_X(\Fscr,\Escr)=\Ext^1(\Fscr,\Escr)=0\}.
\]
It is easy to see that $\Fscr^\perp$ is an abelian subcategory of $\Qcoh(X)$ closed
under direct sums. So it is in particular a Grothendieck category. We will now use
some results by Aidan Schofield, which are unfortunately not officially published.
Proofs can be found in \cite{schofieldup}.
\begin{proposition} \cite{schofieldup}
The inclusion $\Fscr^\perp\subset \Qcoh(X)$ has a left adjoint.
\end{proposition}
Denote the left adjoint to $\Fscr^\perp\r \Qcoh(X)$ by $L$.
Let $p\in X$. There exists
an epimorphism $\phi:\Fscr\r \Oscr_p$. Put $\Fscr'=\ker \phi$, $\Pscr\overset{\text{}}{=}L(\Fscr')$.
\begin{proposition}\cite{schofieldup}
The object $\Pscr$ is a small projective generator for the category~$\Fscr^\perp$. If $\Escr\in \Fscr^\perp$ then $\Hom_X(\Pscr,\Escr)$
is finite dimensional if and only if $\Escr$ is coherent.
\end{proposition}
Put $A=\End_X(\Pscr)$. It follows that  there is an equivalence of categories
\begin{equation}
\label{ref-3.1-20}
\Fscr^\perp\r \Mod(A):\Escr\mapsto \Hom_X(\Pscr,\Escr)
\end{equation}
which is an analogue to \eqref{ref-2.4-15}. With a similar argument as in the discussion thereafter, we obtain an isomorphism
of functors
\[
\Rep_{r,A}\cong \Mscr_{r,d,\Fscr}
\]
and hence by Lemma \ref{ref-2.3.1-5} it follows that $\Ascr_{r,A}$ is split and there is an isomorphism
\[
U_{r,A}\cong M_{r,d,\Fscr}.
\]
\begin{proof}[Proof of Proposition \ref{ref-3.1-19}] There exists some $\Fscr$ such that
$x\in M_{r,d,\Fscr}$. Now let the notation be as above. Then by \eqref{ref-3.1-20}
\[
\End_{X}(p_*V_x)=\End_{A}(V_x).
\]
It now suffices to invoke Lemma \ref{ref-2.2.1-3}.
\end{proof}

\section{Homological identities}
We recall the basic notions regarding Hochschild cohomology. We state the definitions
for graded algebras since this is the generality which we will need later. For a
comprehensive introduction, see \cite{Weibel}. 

Let $B$ be a graded
$k$-algebra and let $M$ be a graded $B$-$B$-bimodule. Construct the
(graded) \emph{Hochschild complex} as
$$C^i(B,M)=\Hom_{\gr(k)}(B^{\otimes i}, M)=\bigoplus_j \Hom_k(B^{\otimes i},M)_j$$
where $\Hom_k(B^{\otimes i},M)_j$ represents the set of $k$-multilinear maps $B^{\otimes i}\to M$ of degree~$j$. The differential is given by 
\begin{multline*}
d_{\text{Hoch}}(f)(r_0,\ldots, r_i)=r_0f(r_1,\ldots,r_i)-f(r_0r_1,\ldots,r_i)+\cdots +(-1)^{i-1} f(r_0,\ldots,r_{i-1}r_i)\\+(-1)^i f(r_0,\ldots,r_{i-1})r_i.
\end{multline*}
The Hochschild cohomology is defined as
$$\HH^i(B,M)=H^iC^*(B,M)$$
and has a decomposition $\HH^i(B,M)=\bigoplus_j \HH^i(B,M)_j$, where the elements of $\HH^i(B,M)_j$ are represented by the cocycles of degree $j$ with $i$ arguments.

The \emph{Hochschild dimension} of a $k$-algebra $B$ is the projective
dimension of $B$ as an $B\otimes_k
B^{\text{op}}$-module. Equivalently, it is the maximum $d$ such that
there exists an $B$-$B$-bimodule $M$ with $\HH^d(B,M)\neq 0$. If $B$
is a finitely generated field extension of $k$, then \cite{Osofsky}
the Hochschild dimension of $B$ equals its transcendence degree over
$k$.

\medskip

In this section we use Hochschild cohomology to compute $\Ext$-groups in base extended categories.
We first recall the following ``change of rings'' result:
\begin{proposition}
Let $\Cscr$ be a $k$-linear Grothendieck category and let $B$ be a $k$-algebra. Then for $M,N \in D(\Cscr_B)$ we have
\begin{equation}
\label{ref-4.1-21}
\RHom_{B\otimes_k B^{\circ}}(B,\RHom_\Cscr(M,N))=\RHom_{\Cscr_B}(M,N).
\end{equation}
\end{proposition}
This proposition is an immediate consequence of the following lemma by setting  $C=B$, $P=B$:
\begin{lemma}
\label{ref-4.2-22}
Let $\Cscr$ be a $k$-linear Grothendieck category, and let $B$, $C$ be $k$-algebras. Then the following
identity holds:
\begin{equation}
\label{ref-4.2-23}
\RHom_{C\otimes_k B^{\circ}}(P,\RHom_\Cscr(M,N))=\RHom_{\Cscr_C}(P\Lotimes_B M,N)
\end{equation}
where $M\in D(\Cscr_B)$, $N\in D(\Cscr_C)$, $P\in D(C\otimes_k B^\circ)$.
\end{lemma}
\begin{proof}
  We may assume that $N$ is fibrant for the standard model
  structure on complexes over $\Cscr_C$ (e.g.\ \cite{Beke}) and that $P$ is
  cofibrant as $C\otimes_k B^\circ$-complex for the projective model
  structure on complexes \cite{Hovey}.  It easy to see that $N$ is
  fibrant as a complex over $\Cscr$ and $P$ is cofibrant as a
  $B^\circ$-complex. In that case we must show
\[
\underline{\Hom}_{C\otimes_k B^{\circ}}(P,\underline{\Hom}_\Cscr(M,N))=\underline{\Hom}_{\Cscr_C}(P\otimes_B M,N).
\]
where $\underline{\Hom}(-,-)$ denotes the morphism complex.
We claim the left and right hand side are the same complex.
It is enough to show this when $P,M,N$ are  objects concentrated in degree zero (with $P$ projective and
$N$ injective). In this
case we must show
\[
\Hom_{C\otimes_k B^{\circ}}(P,{\Hom}_\Cscr(M,N))=\Hom_{\Cscr_C}(P\otimes_B M,N).
\]
Since $P$ is projective, we may reduce to the case that $P$ is free, so that $P=C\otimes_k B^\circ$. Then
we must show
\[
\Hom_{C\otimes_k B^{\circ}}(C\otimes_k B^{\circ},\Hom_\Cscr(M,N))=\Hom_{\Cscr_C}((C\otimes_k B^\circ)\otimes_B M,N),
\]
which is clearly true.
\end{proof}

The following is a useful corollary:
\begin{corollary}
\label{cor:cor1}
Assuming that $M$ is right bounded and $N$ is left bounded, we get a convergent spectral sequence
\[
E_2^{p,q}= \HH^p(B,\Ext^q_{\Cscr}(M,N))\Rightarrow \Ext^{p+q}_{\Cscr_B}(M,N).
\]
\end{corollary}
Below we will need the following consequence:
\begin{corollary}
\label{cor:cor2}
 Assume that  $\Cscr$ has global dimension
one. Assume furthermore $M,N\in \Cscr_B$. Then there is an isomorphism
\[
\HH^0(B,\Hom_\Cscr(M,N))=\Hom_{\Cscr_B}(M,N)
\]
as well as a long
exact sequence
\[
\begin{gathered}
\HH^1(B,\Hom_{\Cscr}(M,N))\r \Ext^1_{\Cscr_B}(M,N)\r \HH^0(B,\Ext^1_{\Cscr}(M,N))
\r
\\
\HH^2(B,\Hom_{\Cscr}(M,N))\r \Ext^2_{\Cscr_B}(M,N)\r \HH^1(B,\Ext^1_{\Cscr}(M,N))
\r\\
\HH^3(B,\Hom_{\Cscr}(M,N))\r \Ext^3_{\Cscr_B}(M,N)\r \HH^2(B,\Ext^1_{\Cscr}(M,N))
\r.
\end{gathered}
\]
In particular, if $\Cscr_B$ also has global dimension one then
\begin{equation}
\label{ref-4.3-24}
\HH^{1+i}(B,\Ext^1_{\Cscr}(M,N))
\cong
\HH^{3+i}(B,\Hom_{\Cscr}(M,N))
\end{equation}
for $i\ge 0$.
\end{corollary}
\begin{proof} Writing $\HH^p(\Ext^q)$ for $\HH^p(B,\Ext^q_\Cscr(M,N))$, the spectral sequence
looks like
\[
\xymatrix{
0&0&0&0\\
\HH^0(\Ext^1)\ar[rrd] & \HH^1(\Ext^1)\ar[rrd] & \HH^2(\Ext^1) & \HH^3(\Ext^1) & \dots\\
\HH^0(\Ext^0) & \HH^1(\Ext^0) & \HH^2(\Ext^0) & \HH^3(\Ext^0) & \dots\\
}.
\]
The conclusion easily follows.
\end{proof}
Finally, here is another corollary we will use.  
\begin{corollary}
\label{ref-4.5-25}
Assume that there is a $k$-algebra morphism $\rho:C\r B$.
Let $N$ be in $D(C\otimes_k B^\circ)$. Then there is a canonical isomorphism
\[
\RHom_{B\otimes_k B^\circ}(B,\RHom_C(B,N))=\RHom_{C\otimes_k C^\circ}(C,N)
\] 
where we have considered $B$ as a $C\otimes_k B^\circ$ module via 
the map $\rho\otimes 1:C\otimes_k B^\circ\r B\otimes_k B^\circ$.
\end{corollary}
\begin{proof}
We apply \eqref{ref-4.1-21}, where $\Cscr=\Mod(C)$, $M=B$.  In this way we get
\begin{align*}
\RHom_{B\otimes_k B^\circ}(B,\RHom_C(B,N))&=\RHom_{B^\circ\otimes_k B}(B^\circ,\RHom_C(B,N))\\
&=\RHom_{C\otimes_k B^\circ}(B,N)\\
&=\RHom_{C\otimes_k C^\circ}(C,N)
\end{align*}
The last equality follows from ``change of rings'' since 
$
(C\otimes_k B^\circ)\Lotimes_{C\otimes_k C^\circ} C=B
$.
\end{proof}

\section{Lifting field actions in the hereditary case}
\label{ref-5-26}
Recall that a $k$-algebra $A$ is defined to be of finite representation type if there are finitely many isomorphism classes of indecomposable left $A$-modules. $A$ is tame if it is not of finite representation type and if the isomorphism classes of indecomposable left A-modules in any fixed dimension are almost all contained in a finite number of 1-parameter families.

Let $A$ be a finite dimensional $k$-algebra which is either tame or of finite representation type and
let $X/k$ a smooth projective curve of genus $g\le 1$. Let $L/k$ be an arbitrary field extension.

The principal application of the results in this section is the fact that the essential images of the
functors
\[
D^b(\mod(A_L))\r D^b(\Mod(A))_L
\]
and
\[
D^b(\coh(X_L))\r D^b(\Qcoh(X))_L
\]
are precisely the objects which have cohomology in $\mod(A_L)\subset \Mod(A_L)=\Mod(A)_L$ in the first case
and in $\coh(X_L)\subset \Qcoh(X_L)=\Qcoh(X)_L$ in the second case.

To be consistent
with the setup in the introduction, we would have preferred to talk about 
$\mod(A)_L$ instead of $\mod(A_L)$ and similarly about $\coh(X)_L$ instead
of $\coh(X_L)$.  Unfortunately, this is incorrect.  If $\Cscr$ is a
$\Hom$-finite abelian category and $L/k$ is an infinite field extension,
then $\Cscr_L$ contains only the zero
object.

In order to be able to describe our results abstractly, we will first discuss
a different notion of base extension for essentially small abelian
categories such as $\mod(A)$, $\coh(X)$ which behaves in the way we
expect. 
\subsection{Base extension for essentially small abelian categories}
 Let $\Cscr$ be an essentially
small abelian category. The category $\Ind \Cscr$ is obtained by formally closing $\Cscr$ under direct 
limits (see e.g. \cite[\S2.2]{LowenVdB1}).  It is well known that $\Ind{\Cscr}$ is a Grothendieck category and, furthermore,
$\Cscr$ can be recovered as $\Fp\Ind{\Cscr}$, the category of finitely presented objects in $\Ind{\Cscr}$.

A Grothendieck category is said to be locally coherent if it is
locally finitely presented (that is: generated by its finitely
presented objects) and the finitely presented objects form an abelian
subcategory. Thus $\Ind \Cscr$ is locally coherent.  Conversely, if
$\Dscr$ is a locally coherent Grothendieck category then
$\Dscr\cong\Ind\Fp\Dscr$.

Now assume that $\Cscr$ is in addition $k$-linear, and let $B$ be a $k$-algebra. If $C\in \Cscr$, then $B\otimes_k C$ is finitely presented in
$(\Ind \Cscr)_B$, and hence $(\Ind \Cscr)_B$ is locally finitely presented. Set
\[
\Cscr_{[B]}=\Fp((\Ind\Cscr)_B).
\]
In good cases, $\Cscr_{[B]}$ will be abelian (or equivalently:
$(\Ind \Cscr)_B$ will be locally coherent). Here are some typical examples:
\begin{examples}
\begin{enumerate}
\item If  $X$ is a $k$-scheme of finite type, then $\coh(X)_{[L]}=\coh(X_L)$ for $L/k$ an arbitrary field extension.
\item If $A$ a finite dimensional $k$-algebra, then $\mod(A)_{[L]}=\mod(A_L)$.
\end{enumerate}
\end{examples}
We need to extend this notion of base extension to the derived setting. 
Assuming that $\Cscr_{[B]}$ is abelian, we will define $D^b(\Cscr)_{[B]}$ as the full subcategory of $D^b(\Ind \Cscr)_B$ whose objects
 have cohomology in $\Cscr_{[B]}$. Thus we have a $2$-Cartesian commutative diagram
\begin{equation}
\label{ref-5.1-27}
\xymatrix{%
D^b(\Cscr_{[B]})\ar[d]_{[F]}\ar@{^(->}[r] & D^b((\Ind \Cscr)_B)\ar[d]^F\\
D^b(\Cscr)_{[B]}\ar@{^(->}[r] & D^b(\Ind \Cscr)_B\\
}
\end{equation}
The full faithfulness of the lower horizontal arrow is by definition, whereas the full faithfulness of the upper arrow is an application of \cite[Prop.\ 2.14]{LowenVdB1}, which asserts that for an essentially small abelian category $\Dscr$
the natural functor
\[
D^b(\Dscr)\r D^b(\Ind \Dscr)
\]
is fully faithful (and its essential image is $D^b_{\Dscr}(\Ind \Dscr)$). We
apply this result with $\Dscr=\Cscr_{[B]}$ since then by construction we have
$\Ind \Dscr=\Ind\Fp ((\Ind\Cscr)_B)
\cong (\Ind\Cscr)_B$.
\subsection{General discussion}
\label{ref-5.2-28}
Let $\Dscr$ be a $k$-linear hereditary category, i.e.\ an abelian category of global dimension one, and let $L/k$ be a field extension. Let $Z\in D^b(\Dscr)_L$. 
In $D^b(\Dscr)$ we have  $Z\cong\bigoplus_i s^i H^{-i}(Z)$, where $s$ denotes the shift functor.
Thus $\End_\Dscr(Z)$ is given by lower triangular
matrices
\[
\small
\begin{pmatrix}
\dots& \dots &\dots &\dots &\dots\\
\dots& \End_\Dscr(H^{-i+1}(Z)) &0 &0 &\dots\\
\dots&\Ext^1_\Dscr(H^{-i+1}(Z),H^{-i}(Z)) &\End_\Dscr(H^{-i}(Z)) & 0&\cdots\\
\dots & 0 &\Ext^1_\Dscr(H^{-i}(Z),H^{-i-1}(Z)) & \End_\Dscr(H^{-i-1}(Z)) &\cdots\\
\dots& \dots &\dots &\dots &\dots\\
\end{pmatrix}
\]
Similarly $\Ext^j_\Dscr(Z,Z)$ is given by
\begin{equation}
\label{ref-5.2-29}
\small
\begin{pmatrix}
\dots& \dots &\dots &\dots &\dots\\
\dots& \Hom_\Dscr(H^{-i+1}(Z),H^{-i+1-j}(Z)) &0 &0 &\dots\\
\dots&\Ext^1_\Dscr(H^{-i+1}(Z),H^{-i-j}(Z)) &\Hom_\Dscr(H^{-i}(Z),H^{-i-j}(Z)) & 0&\cdots\\
\dots & 0 &\Ext^1_\Dscr(H^{-i}(Z),H^{-i-1-j}(Z)) & \End_\Dscr(H^{-i-1}(Z),H^{-i-1-j}(Z)) &\cdots\\
\dots& \dots &\dots &\dots &\dots\\
\end{pmatrix}
\end{equation}
The $L$-bimodule structure on $\Ext^j_\Dscr(Z,Z)$ is given by its
action on $Z$, i.e. by its morphism $L\r \End_\Dscr(Z)$.

\begin{definitions} Let $\Cscr$ be an essentially small $k$-linear abelian category
which satisfies the following additional conditions
for every field extension $L/k$:
\begin{enumerate}
\item[(E1)] $\Cscr_{[L]}$ is abelian.
\item[(E2)] Every object in $\Cscr_{[L]}$ is a direct sum of indecomposable objects.
\end{enumerate}
\label{ref-5.2.1-30} 
We say that $\Cscr$ is of essential dimension $\leq d$ if every
indecomposable object $V$ in $\Cscr_{[L]}$ with ${L}$
algebraically closed can be defined over a finitely generated field
extension of $k$ of transcendence degree $\le d$. More precisely, there exists a finitely generated $L_0/k$ such that $\operatorname{trdeg}_k
L_0 \leq d$ and $V_0\in \Cscr_{[L_0]}$ such that $V\cong L\otimes_{L_0}
V$, where $-\otimes_{L_0} V$ is the unique finite colimit preserving functor such that $L_0\otimes_{L_0}V=V$. The minimal such $d$ valid for all $V$ is called the essential
dimension $\operatorname{ess} \Cscr$ of $\Cscr$. If such $d$ does not
exist then $\operatorname{ess}\Cscr=\infty$.
\end{definitions}
Note that (E2) holds if $\Cscr_{[L]}$ is $\Hom$-finite.

\begin{remarks}
\label{ref-5.2.2-31}
\begin{itemize}
\item If $\Cscr=\mod(A)$, where $A$ is a finite dimensional hereditary
$k$-algebra, then it follows from the classification of indecomposable representations
for Dynkin and extended Dynkin quivers \cite{Gabriel2,Ringel4} as well
as the existence of arbitrarily large moduli spaces in the other cases (e.g.\ Proposition \ref{ref-2.5.7-17})
that 
$\ess \Cscr=0,1,\infty$ depending on whether $\Cscr$
is of finite representation type, tame or wild.
\item If $\Cscr=\coh(X)$, where $X$ is a projective smooth curve/$k$, then it follows from the Grothendieck
classification of indecomposable coherent sheaves on $\PP^1$ and the corresponding (much harder) classification by Atiyah
for elliptic curves, as well as the existence of arbitrarily large moduli spaces in the other cases (e.g.\ \S\ref{ref-3-18}),
that 
$\ess\Cscr$ is $1$ if $X$ is $\PP^1$ or an elliptic curve and $\infty$
otherwise. 
\item It is not clear to us if there can be examples with $\ess \Cscr$ strictly
bigger than~$1$ but finite. In the standard algebraic and geometric cases this is probably excluded
by the tame-wild dichotomy.
\end{itemize}
\end{remarks}
\begin{theorems} \label{ref-5.2.3-32}  Let $\Cscr$ be an essentially small
  $k$-linear abelian category satisfying (E1)(E2) above, and assume in addition that $\Ind \Cscr$ is hereditary. 
Consider the usual 
  forgetful functor $[F]:D^b(\Cscr_{[L]})\r D^b(\Cscr)_{[L]}$ for an arbitrary field field extension
$L/k$.
\begin{itemize}
\item If $\Cscr$ has essential dimension $\le 2$, $[F]$ is essentially surjective.
\item If $\Cscr$ has essential dimension $\le 1$, $[F]$ is in addition full.
\item If $\Cscr$ has essential dimension $0$, $[F]$ is an equivalence
of categories. 
\end{itemize}
\end{theorems}
\begin{proof} The three parts of the theorem are similar, so we will only prove the
first assertion. Assume thus $\ess \Cscr\le 2$. 

Set $\Dscr=\Ind \Cscr$ and let $Z\in D^b(\Cscr)_{[L]}\subset D^b(\Dscr)_L$.
By \eqref{ref-5.1-27} it is sufficient to prove that $Z$ is in the essential image
of $D^b(\Dscr_L)$.

 The
lower triangular structure of the matrix 
\eqref{ref-5.2-29}
equips $\Ext^j_\Dscr(Z,Z)$ in a natural way
with a two-step filtration stable under the left and right $\End_\Dscr(Z)$-action.
Hence $\Ext^j_\Dscr(Z,Z)$ is a two-step filtered $L$-bimodule and the
associated quotients are sums of 
\[
\Ext^l_\Dscr(U,V)
\]
where $U,V$ are among the $H^i(Z)$. In particular $U,V$ are objects in $\Cscr_{[L]}$.

To prove essential surjectivity, we have to show the vanishing of
\[
\HH^n(L,\Ext^{-n+2}_\Dscr(Z,Z))
\]
for $n\ge 3$ (see Proposition \ref{ref-9.1.1-48} below). Using the above filtration, it is sufficient to show the vanishing of 
\begin{equation}
\label{ref-5.3-33}
\HH^n(L,\Ext^l_\Dscr(U,V))
\end{equation}
for all $l$, $n\ge 3$ and for $U,V\in \Cscr_{[L]}$.  Let $\bar{L}$ be the algebraic closure of $L$.
We have
\begin{align*}
\HH^n(\bar{L},\Ext^l_{\Dscr}(\bar{L}\otimes_L U,\bar{L}\otimes_L V))&
=\HH^n(\bar{L},\Hom_L(\bar{L},\Ext^l_{\Dscr}(U, \bar{L}\otimes_L V)))\\
&=\HH^n(L,\Ext^l_{\Dscr}(U, \bar{L}\otimes_L V)))\qquad (\text{Corollary \ref{ref-4.5-25}}).
\end{align*}
Since $V$ is a direct summand of $\bar{L}\otimes_L V$, it suffices to prove that \eqref{ref-5.3-33}
vanishes in the case that $L$ is algebraically closed.
It is clear that we may assume
in addition that  $U,V$ are indecomposable.  Since $\ess \Cscr\le 2$ we may
write $U=L\otimes_{L_0} U_0$ where $L_0/k$ is a finitely generated field
extension of transcendence degree at most two and $U_0$ is in $\Dscr_{L_0}$. We
then find $\Ext^l_\Dscr(U,V)=\Hom_{L_0}(L,\Ext^l_\Dscr(U_0,V))$ and so
\begin{align*}
\HH^n(L,\Ext^l_\Dscr(U,V))&=
\HH^n(L,\Hom_{L_0}(L,\Ext^l_\Dscr(U_0,V)))\\
&=\HH^n(L_0,\Ext^l_\Dscr(U_0,V))\qquad \text{(Corollary \ref{ref-4.5-25})}\\
&=0 \qquad \text{(since $n\ge 3$)}\qed
\end{align*}
\def\qed{}\end{proof}
\section{Counterexamples to lifting in the hereditary case}
In this section we prove a non-lifting theorem in the hereditary case. In contrast to the previous section we use
standard base extension for abelian categories as defined in the introduction.
\begin{proposition} \label{ref-6.1-34}
Let either $\Cscr$ be $\Mod(kQ)$ with $Q$ be a connected finite ``wild'' quiver (i.e.\ $Q$ is not Dynkin or extended Dynkin) 
or else $\Cscr$ be $\Qcoh(X)$ with $X$ a curve of genus $\ge 2$.
Then there exists a finitely generated field extension $L/k$
  of transcendence degree $3$ together with an object $Z\in D^b(\Cscr)_{L}$ which is not a direct summand of an object in 
the essential image of the forgetful functor $F:D^b(\Cscr_{L})\r D^b(\Cscr)_{L}$. We may in addition
assume that the cohomology of $Z$ lies in $\mod(kQ_L)$ or $\coh(X_L)$, depending on the situation.
\end{proposition}
The proof will occupy the rest of this section.   
For simplicity we will assume that~$Q$ is not the quiver with one vertex and two loops, as this 
case needs a more general argument which we will give in Appendix \ref{ref-A-69}.

We first give a necessary and sufficient condition for an object in $D^b(\Cscr)_L$ to be in the
essential image of $F$ assuming that, after forgetting the $L$-action, it has the form $Z=U\oplus
sV\in D^b(\Cscr)$ for $U,V\in \Cscr$. We do this by specializing the general formulas
from \S\ref{ref-5.2-28}. We have
\[
\Lambda\overset{\text{def}}{=}\End_\Cscr(Z)=
\begin{pmatrix}
\End_\Cscr(U)&0\\
\Ext^1_\Cscr(U,V) & \End_\Cscr(V)
\end{pmatrix}.
\]
An $L$-action on $Z$ is a $k$-algebra morphism 
\def\triv{\operatorname{triv}}
\[
\phi:L\r \Lambda.
\]
We may write
\[
\phi=\begin{pmatrix}
\phi_{11}&0\\
\phi_{21}&\phi_{22}
\end{pmatrix}
\]
where $\phi_{11}$, $\phi_{22}$ represent an action of $L$ on $U$ and
$V$ respectively, so that $(U,\phi_{11})$ and $(V,\phi_{22})$ are in
$\Cscr_L$.  We will denote by $(Z,\phi)$ the corresponding object of
$D^b(\Cscr)_L$.

The condition that $\phi$ is compatible with multiplication 
yields
\begin{align*}
\begin{pmatrix}
\phi_{11}(l_1l_2)&0\\
\phi_{21}(l_1l_2)&\phi_{22}(l_1l_2)
\end{pmatrix}
&=
\begin{pmatrix}
\phi_{11}(l_1)&0\\
\phi_{21}(l_1)&\phi_{22}(l_1)
\end{pmatrix}
\begin{pmatrix}
\phi_{11}(l_2)&0\\
\phi_{21}(l_2)&\phi_{22}(l_2)
\end{pmatrix}\\
&=\begin{pmatrix}
\phi_{11}(l_1)\phi_{11}(l_2)&0\\
\phi_{22}(l_1)\phi_{21}(l_2)+\phi_{21}(l_1)\phi_{11}(l_2)
&\phi_{22}(l_1)\phi_{22}(l_2)
\end{pmatrix}
\end{align*}
In other words,
\[
\phi_{21}:L\r \Ext^1_\Cscr(U,V)
\]
must be a $k$-derivation for the $L$-bimodule structure on $\Ext^1_\Cscr(U,V)$
obtained from the $L$-structures on $U$ and $V$. 
\begin{lemma}
\label{ref-6.2-35} $(Z,\phi)$ as above is in the essential image of $F$
if and only if $\phi_{21}$ is trivial in
\[
\HH^1(L,\Ext^1_\Cscr(U,V)).
\]
\end{lemma}
\begin{proof}
We will write $\phi_{\triv}$ for the trivial action on $Z$ coming 
from the given $L$-action on $U$, $V$ (so that $\phi_{21}=0$).

Assume that $(Z,\phi)$ is in the essential image of $F$. In other words, there exists $Y\in D^b(\Cscr_L)$
such that 
\begin{equation}
\label{ref-6.1-36}
(Z,\phi)\cong FY.
\end{equation}
Since $\Cscr$ is either $\Mod(kQ)$ or $\Qch(X)$, $\Cscr_L$ is of the same type (e.g. $\Mod(kQ)_L=\Mod(LQ)$), hence it is hereditary as well and we have in $D^b(\Cscr_L)$
\[
Y\cong \bigoplus_n s^nH^{-n}(Y).
\]
Furthermore, in $\Cscr_L$ we have 
\[
H^{-n}(Y)=H^{-n}(Z)=
\begin{cases}
U&\text{if $n=0$}\\
V&\text{if $n=1$}\\
0&\text{otherwise}.
\end{cases}
\]
Thus $Y$, considered as an element of $D^b(\Cscr)_L$, is precisely $(Z,\phi_{\triv})$.
In other words $(Z,\phi)\cong FY$ for some $Y\in D^b(\Cscr_L)$ iff there is an
isomorphism in $D^b(\Cscr)_L$
\begin{equation}
\label{ref-6.2-37}
\pi:(Z,\phi)\cong (Z,\phi_{\triv}).
\end{equation}
We may view $\pi$ as a unit in $\Lambda$, and the condition that
$\pi$ is compatible with the $L$-action may be expressed as 
\[
\pi\phi(l)=\phi_{\triv}(l)\pi
\]
for all $l\in L$, i.e.
\begin{equation}
\label{ref-6.3-38}
\phi(l)=\pi^{-1}\phi_{\triv}(l)\pi.
\end{equation}
We now write all conditions explicitly: we have
\[
\phi=\begin{pmatrix}
\phi_{11}&0\\
\phi_{21}&\phi_{22}
\end{pmatrix}
\]
\[
\phi_{\triv}=\begin{pmatrix}
\phi_{11}&0\\
0&\phi_{22}
\end{pmatrix}
\]
\[
\pi
=
\begin{pmatrix}
\pi_{11}&0\\
\pi_{21}&\pi_{22}
\end{pmatrix}
\]
\[
\pi^{-1}
=
\begin{pmatrix}
\pi_{11}^{-1}&0\\
-\pi_{22}^{-1}\pi_{21}\pi_{11}^{-1}&\pi_{22}^{-1}
\end{pmatrix}
\]
and condition \eqref{ref-6.3-38} translates into
\[
\begin{pmatrix}
\phi_{11}&0\\
\phi_{21}&\phi_{22}
\end{pmatrix}
=
\begin{pmatrix}
\pi^{-1}_{11} \phi_{11} \pi_{11} & 0\\
-\pi^{-1}_{22} \pi_{21} \pi^{-1}_{11} \phi_{11} \pi_{11}+\pi^{-1}_{22}\phi_{22}\pi_{21}
&
\pi^{-1}_{22}\phi_{22}\pi_{22}
\end{pmatrix}
\]
which yields 
\begin{align*}
\pi_{11}\phi_{11}&=\phi_{11}\pi_{11}\\
\pi_{22}\phi_{22}&=\phi_{22}\pi_{22}\\
\phi_{21}&=-\pi^{-1}_{22} \pi_{21} \pi^{-1}_{11} \phi_{11} \pi_{11}+\pi^{-1}_{22}\phi_{22}\pi_{21}.
\end{align*}
Taking into account the commutation relation given by the first
two relations, the last one can be written as
\[
\phi_{21}=\phi_{22} (\pi^{-1}_{22}\pi_{21})-(\pi^{-1}_{22}\pi_{21})\phi_{11}.
\]
In other words, the existence of $\pi$ implies that $\phi_{21}$ is an
inner derivation, which means precisely that $\phi_{21}$ is a coboundary in the Hochschild complex. It is easy to see that this implication is 
reversible. 
\end{proof}
\begin{lemma} 
\label{ref-6.3-39}If $(Z,\phi) \in D^b(\Cscr)_L$ as above is not in the essential image of $F$,
  then it is also not 
 a direct summand of an object in the
  essential image of $F$.
\end{lemma}
\begin{proof} Assume that $(Z,\phi)$ is not in the essential image of $F$ but there exist
$W\in D^b(\Cscr)_L$, $Y\in D^b(\Cscr_L)$ such that 
\begin{equation}
\label{ref-6.4-40}
(Z,\phi)\oplus W\cong FY.
\end{equation}
Note that the truncation functors $\tau_{\le i}$, $\tau_{\ge i}$
commute with $F$. Applying $\tau_{\le 0}\tau_{\ge -1}$ to
\eqref{ref-6.4-40} we obtain the existence of objects $(Z',\phi')\in
D^b(\Cscr)_L$ ($Z'=U'\oplus sV'$, $U',V'\in \Cscr$),  $Y'\in
D^b(\Cscr_L)$ such that $(Z\oplus Z',\phi+\phi')\cong FY'$. This means
that  $\phi_{21}+\phi'_{21}$ is zero in $\HH^1(L,\Ext^1_{\Cscr}(U\oplus U',V\oplus V'))$.
However, it is clear that $\phi_{21}$ and $\phi'_{21}$ land in different summands
of $\HH^1(L,\Ext^1_{\Cscr}(U\oplus U',V\oplus V'))$. So this implies $\phi_{21}=0$,
which is in contradiction with the fact that $(Z,\phi)$ is not in the essential image of $F$.
\end{proof}
\begin{proof}[Proof of Proposition \ref{ref-6.1-34}]
Now we recall that by \eqref{ref-4.3-24} we have
\[
\HH^1(L,\Ext^1_\Cscr(U,V))\cong \HH^{3}(L,\Hom_\Cscr(U,V)).
\]
Let $\Cscr_{0,L}=\mod(kQ_L)$ or $\coh(X_L)$, depending on whether $\Cscr$ is equal to $\Mod(kQ)$ or $\Qcoh(X)$.
To construct $Z$ as in the statement of the proposition, it suffices by Lemmas \ref{ref-6.2-35}
and \ref{ref-6.3-39}
to find $L$ and $U,V\in \Cscr_{0,L}$ such that
$\HH^3(L,\Hom_{\Cscr}(U,V))\neq 0$. We will in fact produce a finitely generated field extension $L/k$
of transcendence degree $3$ and $U\in \Cscr_{0,L}$ such that $\End_{\Cscr}(U)=L$, and let $V=U$.
This will do what
we want by the Hochschild-Kostant-Rosenberg theorem \cite{HKR}.

Let us first consider the case $\Cscr=\Mod(kQ)$, $Q$ not the two-loop quiver.
  Choose $\alpha$ indivisible, and in the fundamental region as
in Proposition \ref{ref-2.5.7-17}, such that $\dim M_{\alpha,\lambda,Q}\ge 3$. Let $x$ be the generic point of
a three dimensional irreducible subvariety of $M_{\alpha,\lambda,Q}$, and put $L=k(x)$. Let $U=V_x$ as in Proposition \ref{ref-2.5.2-11}. 
Then according to \eqref{ref-2.3-12} we have indeed $\End_{kQ}(U)=L$.

Next consider $\Cscr=\Qcoh(X)$. We choose $r,d$, $\gcd(r,d)=1$ such that $\dim M_{r,d}\ge 3$. Then
we proceed as in the case $\Cscr=\Mod(kQ)$, but now using Proposition \ref{ref-3.1-19}.
\end{proof}
\section{Counterexamples to lifting in the geometric case}
In this section we will prove the following result:
\begin{theorem}
\label{ref-7.1-41}
Let $Y/k$ be a smooth connected projective variety which is not a point, a projective line or an elliptic curve. Then there exists a finitely generated field extension $L/k$
  of transcendence degree $3$ together with an object $Z\in D^b(\Qcoh(Y))_L$ which is not a direct summand of an object in 
the essential image of the forgetful functor $F:D^b(\Qcoh(Y)_L)\r D^b(\Qcoh(Y))_L$. We may in addition
assume that the cohomology of $Z$ lies in $\coh(Y_L)$.
\end{theorem}
\begin{proof}
If $Y$ is a curve then this follows from Proposition \ref{ref-6.1-34}, so we may assume $\dim Y\ge 2$. We start by considering the case
$Y=\PP^d$, $d\ge 2$. Put $T=\Oscr_Y\oplus \Oscr_Y(1)$ and $A=\End_Y(T)^\circ$. Then $T$ is a partial tilting object and we have functors
\[
D^b(\Mod(A))\xrightarrow[i]{T\Lotimes_A-} D^b(\Qcoh(Y))\xrightarrow[j]{\RHom_Y(T,-)} D^b(\Mod(A))
\]
such that $ji$ is the identity. We can define analogous functors, also denoted by $i$ and $j$, on $D^b(\Mod(A_L))$, $D^b(\Mod(A)_L)$, etc., and $i$, $j$ commute with the functor $F$. Now $A$ is the path algebra of a Kronecker quiver with $d+1\ge 3$ arrows, and so according to Proposition \ref{ref-6.1-34} there exists
a finitely generated field extension $L/k$ of transcendence degree three and an object $Z_0\in D^b(\Mod(A))_L$ with cohomology in $\mod(A_L)$ which is not a direct summand of the image
of an object in $D^b(\Mod(A)_L)$. Put $Z_1=i(Z_0)\in D^b(\Qcoh(Y))_L$. Clearly $Z_1$ has cohomology in $\coh(Y_L)$.
If there exists $Z'\in D^b(\Qcoh(Y))_L$, $Y\in D^b(\Qcoh(Y)_L)$ such that $Z_1\oplus Z'=FY$, 
then applying $j$ we get $Z_0\oplus j(Z')=F(j(Y))$ in $D^b(\Mod(A)_L)$, which we had excluded. This finishes the proof in the case $Y=\PP^d$.

Now let $Y$ be general, choose a finite (necessarily flat) map $\pi:Y\r \PP^d$, and let $Z_1\in D^b(\Qcoh(\PP^d))_L$ be as above.
Set $Z=\pi^\ast Z_1$. 
 If there exists $Z'\in D^b(\Qcoh(Y))_L$, $Y\in D^b(\Qcoh(Y)_L)$ such that $Z\oplus Z'=FY$ 
then applying $R\pi_\ast$ we get $R\pi_\ast\pi^\ast Z_1\oplus R\pi_\ast Z'=F(R\pi_\ast Y)$. Now $R\pi_\ast\pi^\ast Z_1=\pi_\ast \Oscr_Y\otimes_{\Oscr_{\PP^d}} Z_1$ and $\Oscr_{\PP^d}$
is a direct summand of $\pi_\ast\Oscr_Y$ in $\coh(\PP^d)$ (by the trace map).  Hence $Z_1$ is a direct summand of $R\pi_\ast\pi^\ast Z_1$ in $D^b(\Qcoh(\PP^d))_L$, and thus
$Z_1$ is also a direct summand of $F(R\pi_\ast Y)$ in $D^b(\Qcoh(\PP^d))_L$. This is a contradiction with the choice of $Z_1$.
\end{proof}
\begin{example}
  \label{ref-7.2-42} Assume that $Y=\PP_k^3$, $L=k(x,y,z)$ and let
  $p:\PP^3_L\r \PP_k^3$ be obtained by base extension of the structure
  map $\Spec L\r \Spec k$. If we construct $Z_0$ using the object
  $V_\eta$ in Example \ref{ref-2.5.6-16} (see the proof of
  Proposition \ref{ref-6.1-34}) then we find that, after forgetting the $L$-structure, $Z\cong p_\ast R\oplus
  p_\ast sR$,  where $R$ is given by
\[
\operatorname{cone}(\Oscr_{\PP^3_L}(-1)^3 \xrightarrow{(Tx-X,Ty-Y,Tz-Z)} \Oscr_{\PP^3_L})
\]
and $T,X,Y,Z$ are homogeneous coordinates on $\PP^3$.
\end{example}
\section{Non-Fourier-Mukai functors}
Below $X$, $Y$ are smooth connected projective schemes over $k$, although we could get by with substantially less.
Let $i_\eta:\eta\r X$ be the generic point and let $L=k(\eta)$ be the function field of $X$. Assume that
$D^b(\Qcoh(Y))_L$ contains an object $Z$ which is not in the essential image of $D^b(\Qcoh(Y)_L)$. Define
the exact functor $\widetilde{\Psi} : D(\Qcoh(X))\r D(\Qcoh(Y))$ as the composition
\[
\xymatrix{
\widetilde{\Psi}:& D(\Qcoh(X)) \ar[r]^-{i^*_{\eta}} &  D(L) \ar[r]^-{\psi} & D(\Qcoh(Y))
}
\]
where $\psi:D(L) \xrightarrow{} D(\Qcoh(Y))$
 is the unique additive functor commuting with shifts and coproducts
which sends $L$ to $Z$, and is determined on morphisms by the structure of $Z$ as an $L$-object. This functor is exact, because $L$ is a field. 
By construction, $\tilde{\Psi}$ commutes with coproducts.
Let $\Psi:\Perf(X)\r D^b(\Qcoh(Y))$ be the restriction of $\widetilde{\Psi}$ to $\Perf(X)=D^b(\coh(X))$. 
\begin{theorem} \label{ref-8.1-43}
The functor
\[
\Psi:\Perf(X)\r D^b(\Qcoh(Y))
\]
as defined above is not the restriction of a Fourier-Mukai functor $D(\Qcoh(X))\r D(\Qcoh(Y))$ associated
to an object in $D(\Qcoh(X\times Y))$.
\end{theorem}
Taking $Y$, $Z$, $L$ as in Theorem \ref{ref-7.1-41}, and letting $X$ be a smooth projective model for $L$, i.e. such that $K(X)=L$, gives a counterexample to Proposition \ref{ref-A-0} in the introduction
if we drop the condition that $\Psi$ is fully faithful. By taking $Z$ as in Example \ref{ref-7.2-42} we get a counterexample where $X=Y=\PP^3$.

\medskip

We will give the proof of Theorem \ref{ref-8.1-43} below, after some preparatory lemmas.
\begin{lemma}
\label{ref-8.2-44}
Assume that
\[
\Phi:D(\Qcoh(X))\r D(\Qcoh(Y))
\]
is an exact functor, commuting with coproducts, whose restriction to $\Perf(X)$ is naturally equivalent
to $\Psi$. Then $\Phi$ is naturally equivalent to $\widetilde{\Psi}$.
\end{lemma}
\begin{proof} We first claim that $\Phi$ factors uniquely as 
\begin{equation}
\label{eq:factorization}
D(\Qcoh(X))\xrightarrow{i^\ast_\eta} D(L)\xrightarrow{\phi} D(\Qcoh(Y)),
\end{equation}
where $\phi$ is an exact functor commuting with coproducts. To see
this note that the first arrow $i^\ast_\eta$ is a Verdier localization
at the the full subcategory $\Cscr$ of $D(\Qcoh(X))$ spanned by objects $\Mscr$
such that $i_\eta^\ast\Mscr=0$. We claim that $\Cscr$ is generated by objects
which are compact in $D(\Qcoh(X))$ (i.e.\ perfect complexes). 

 By following
the inductive procedure of the proof of \cite[Prop. 2.5]{Neeman1} (also
\cite[Thm 3.1.1]{BondalVdB}) one reduces this claim to the case that $X=\Spec R$.
In that case $D(\Qcoh(X))=D(R)$ and $L$ is the quotient field of $R$. 
The complexes $M(s)=R\xrightarrow{s\times} R$, $s\in R-\{0\}$  
are generators for the kernel of $D(R)\xrightarrow{L\otimes_R-} D(L)$.
Indeed if $N$ is in this kernel and is right orthogonal to all shifts of all $M(s)$ for $s\in R-\{0\}$ then for 
all such $s$ and all $i$ the map $H^i(N)\xrightarrow{s\times} H^i(N)$ is an isomorphism. On the other hand
since $H^i(N)$ is annihilated by $L$, a non-zero element of $H^i(N)$ is annihilated
by some $s\in R-\{0\}$ which impossible. Thus $H^i(N)=0$ for all $i$ and hence $N=0$.

Since the restriction of $\Phi$ to the perfect complexes in $\Cscr$ is equal to $\Psi$
and since $\Psi$ annihilates such complexes and 
furthermore  since $\Phi$ preserves coproducts, $\Phi$ vanishes on $\Cscr$. Thus the asserted
factorization \eqref{eq:factorization} follows.

Thus it suffices to prove that $\psi$ and $\phi$ are naturally equivalent, given that
we know that $\psi\circ i^\ast_\eta$ and $\phi\circ i^\ast_\eta$ are naturally equivalent when restricted to $\Perf(X)$.
 For this 
we must show that $\psi(L)$ and $\phi(L)$ are isomorphic as $L$-objects in $D(\Qcoh(Y))$.

Now we have $L=i^\ast_\eta(\Oscr_X)$, and so
\[
\psi(L)=(\psi\circ i^\ast_\eta)(\Oscr_X)\cong (\phi\circ i^\ast_\eta)(\Oscr_X)=\phi(L).
\]
So we certainly have an isomorphism $\sigma:\psi(L)\cong \phi(L)$ in $D(\Qcoh(Y))$. To prove that this isomorphism is
compatible with the $L$-structure, we observe that any map $f:L\r L$ is of the form
$i^\ast_\eta(g)\circ i^\ast_\eta(h)^{-1}$ where $g,h$ are morphisms $\Oscr_X(-nE)\r \Oscr_X$ with $E$ an ample divisor and 
$h$ non-zero. Thus we get a diagram
\begin{equation}
\label{ref-8.1-45}
\xymatrix{
\psi(L)\ar@/^2em/[rrr]^\sigma
\ar[dd]_{\psi(f)}\ar@{=}[r] 
&(\psi\circ i^\ast_\eta)(\Oscr_X)\ar[r]^{\cong} \ar[d]|{(\psi \circ i^\ast_\eta)(h)^{-1}}
&(\phi\circ i^\ast_\eta)(\Oscr_X) \ar@{=}[r]\ar[d]|{(\phi \circ i^\ast_\eta)(h)^{-1}}
&\phi(L)\ar[dd]^{\phi(f)}
\\
&(\psi\circ i^\ast_\eta)(\Oscr_X(-nE))\ar[r]^{\cong} \ar[d]|{(\psi \circ i^\ast_\eta)(g)} 
 &(\phi\circ i^\ast_\eta)(\Oscr_X(-nE))\ar[d]|{(\phi \circ i^\ast_\eta)(g)}\\
\psi(L)\ar@{=}[r]\ar@/_2em/[rrr]_\sigma &(\psi\circ i^\ast_\eta)(\Oscr_X)\ar[r]_{\cong} &(\phi\circ i^\ast_\eta)(\Oscr_X) \ar@{=}[r]&\phi(L)
}
\end{equation}
where:
\begin{enumerate}
\item the leftmost rectangle is commutative, since it is obtained by applying $\psi$ to
$f=i^\ast_\eta(g)\circ i^\ast_\eta(h)^{-1}$;
\item the rightmost rectangle is commutative for the same reason;
\item the lower middle rectangle is commutative, since it is obtained from the natural
isomorphism $\psi\circ i^\ast_\eta\cong \phi\circ i^\ast_\eta$;
\item the upper middle rectangle is commutative, since it is obtained from inverting the
vertical arrows in the 
commutative diagram
\[
\xymatrix{
(\psi\circ i^\ast_\eta)(\Oscr_X)\ar[r]^{\cong} 
&(\phi\circ i^\ast_\eta)(\Oscr_X) \\
(\phi\circ i^\ast_\eta)(\Oscr_X(-nE))\ar[u]^{(\psi \circ i^\ast_\eta)(h)} \ar[r]_{\cong}&(\phi\circ i^\ast_\eta)(\Oscr_X(-nE)) \ar[u]_{(\phi \circ i^\ast_\eta)(h)}.
}
\]
\end{enumerate}
It follows that the outer rectangle in \eqref{ref-8.1-45} is commutative, and hence $\sigma$ is indeed
compatible with the $L$-structure.
\end{proof}

\begin{lemma}
\label{ref-8.3-46}
Assume that
\[
\Phi:D(\Qcoh(X))\r D(\Qcoh(Y))
\]
is a Fourier-Mukai functor. Then the $L$ object $(\Phi\circ i_{\eta\ast})(L)$ in $D(\Qcoh(Y))$
lies in the essential image of $F:D(\Qcoh(Y)_L)\r D(\Qcoh(Y))_L$.
\end{lemma}
\begin{proof}
Assume that $\Phi$ is isomorphic to the Fourier-Mukai functor $\Phi_V$ with kernel $V\in D(\Qcoh(X\times Y))$, i.e. $\Phi_V = Rp_{2*}(V\stackrel{L}{\otimes}Lp_1^*(\cdot))$. 
Consider the object $V_\eta\in D(\Qcoh(\Spec L\times Y))$ given by  $V_\eta=(i_{\eta}\times\id)^*V$. Below we show that there is a natural isomorphism
\begin{equation}
\label{ref-8.2-47}
\Phi_V\circ i_{\eta\ast}\cong \Phi_{V_\eta}
\end{equation}
as functors $D(L)\r D(\Qcoh(Y))$.

So it suffices to show that the $L$ object $\Phi_{V_\eta}(L)=Rp_{2\ast} V_\eta$ in $D(\Qcoh(Y))$ lies in the essential image of $F$. Now 
there is a canonical identification $c:\Qcoh(\Spec L\times Y))\r \Qcoh(Y)_L$ which fits in 
a  commutative diagram
\[
\xymatrix{
D(\Qcoh(\Spec L\times Y))\ar[d]_{c}\ar[r]^-{Rp_{2\ast}} & D(\Qcoh(Y))_L\ar@{=}[d]\\
D(\Qcoh(Y)_L)\ar[r]_F& D(\Qcoh(Y))_L.
}
\]
Thus we find $Rp_{2\ast} V_\eta=F(cV_\eta)$ which proves what we want.

\medskip

Now we verify \eqref{ref-8.2-47}. Consider the morphisms
\[
\xymatrix{
&& D(\Qcoh(X\times Y))\ar@/^7mm/[ddrr]^{Rp_{2*}} \\
&& D(\Qcoh(\Spec L\times Y)) \ar[u]^-{R(i_{\eta}\times \id)_*} \ar@/^5mm/[drr]^{R{{p}}_{2*}} 
\\
D(\Qcoh(X))\ar@/^5mm/[uurr]^-{Lp_1^*} & D(L)\ar@/^4mm/[ur]^-{L{{p}}_1^*} \ar[l]_-{i_{\eta\ast}}& &
& D(\Qcoh(Y))
}.
\]
We have
\begin{align*}
(\Phi_V\circ i_{\eta\ast})(-) &= Rp_{2*}(Lp_1^*(i_{\eta\ast}(-))\stackrel{L}{\otimes} V) \\
&= Rp_{2*}((i_{\eta}\times\id)_*(Lp_1^*(-))\stackrel{L}{\otimes} V) \\
&= Rp_{2*} (i_{\eta}\times\id)_* (Lp_1^*(-)\stackrel{L}{\otimes} (i_{\eta}\times\id)^*V) \\
&= Rp_{2*} (Lp_1^*(-)\stackrel{L}{\otimes} (i_{\eta}\times\id)^*V)\\
&=\Phi_{V_\eta}(-)
\end{align*}
The second equality is flat base change for $p_1:X\times Y\r X$. The third equality is the projection formula \cite[Prop.\ 3.9.4]{Lipman} for $i_\eta\times \id$.
\end{proof}
\begin{proof}[Proof of Theorem \ref{ref-8.1-43}] Assume that that $\Psi$ is the restriction of a Fourier-Mukai functor $\Phi:D(\Qcoh(X))\r D(\Qcoh(Y))$. According to
Lemma \ref{ref-8.2-44} we have $\Phi\cong \widetilde{\Psi}$. According to Lemma \ref{ref-8.3-46} $(\Phi\circ i_{\eta\ast} )(L)\cong (\widetilde{\Psi}\circ i_{\eta\ast})(L)$ is in the essential image
of $D(\Qcoh(Y))_L$. But since $(\widetilde{\Psi}\circ  i_{\eta\ast})(L)=(\psi \circ i^\ast_\eta\circ i_{\eta\ast})(L)=\psi(L)=Z$, this is a contradiction.
\end{proof}
\section{Lifting using $A_\infty$-actions}
From now on we only assume that $k$ is a field. We will prove Proposition \ref{ref-B-1} stated in the introduction.
The results from this section were also used in the proof of Theorem \ref{ref-5.2.3-32}.
For the benefit of the reader we will provide some preliminary material concering $A_\infty$-actions.
\subsection{Introduction}
A graded category $\Ascr$ is a category enhanced in the category of
graded $k$-vector spaces. To stress the grading we will sometimes write
$\underline{\Hom}_\Ascr(-,-)$ to denote the $\Hom$-spaces. We denote
the part of degree zero of $\underline{\Hom}_\Ascr(-,-)$ by $\Hom_\Ascr(-,-)$.

Let
$\Cscr$ be a $k$-linear Grothendieck category. The category of complexes over~$\Cscr$ (denoted by~$C(\Cscr)$) is
a DG-category, and, in particular, a graded category. To simplify the notation we write $\underline{\Hom}_{\Cscr}$
for $\underline{\Hom}_{C(\Cscr)}$ to denote the morphism complex, and similarly for $\Hom_\Cscr$ for $\Hom_{C(\Cscr)}$.
Let $B$ be a
DG-algebra over $k$ (at first reading one may assume that $B$ is just
an algebra, concentrated in degree zero). We define the
DG-category $C(B,\Cscr)$ as the category of pairs $(M,\rho_M)$, where $M\in C(B,\Cscr)$ and $\rho_M:B\to\underline{\Hom}_{\Cscr}(M,M)$ is a DG-algebra morphism giving the $B$-action on $M$. We put
\[
D(B,\Cscr)=Z^0(C(B,\Cscr))[\text{Qis}^{-1}].
\]
The construction of $D(B,\Cscr)$ represents no set-theoretic difficulties since
it may be obtained from a model structure on $C(B,\Cscr)$ \cite[Prop.\ 5.1]{lowenvdb2}\footnote{The
proof of this result is based on \cite{Beke}.}. $D(B,\Cscr)$ can be identified with $D(\Cscr_B)$ when $B$ is concentrated in degree zero.

If $\Ascr$ is an arbitrary graded category and $B$ is a graded $k$-algebra, 
we may define the category $\Ascr_{B}$ whose objects are the objects in $\Ascr$
equipped with a $B$-action. 

Let us go back to the case of a DG-algebra $B$ over $k$. There is an obvious functor
\[
F:D(B,\Cscr)\r D(\Cscr)_{H^\ast(B)},
\]
where $H^*(B)$ is the graded $k$-algebra $\oplus_{i\in\ZZ}H^i(B)$, with the multiplication induced by the multiplication on $B$. The functor $F$ is obtained by noticing that $\rho_M:B\to \underline{\Hom}_{D(\Cscr)}(M,M)$ factors through $H^*(B)$ since coboundaries are homotopic to zero.

Below we give proofs of the following results:
\begin{propositions} 
\label{ref-9.1.1-48}
Let $M\in D(\Cscr)_{H^\ast(B)}$ be such that for $n\ge 3$
\[
\HH^n(H^\ast(B),\Ext^\ast_\Cscr(M,M))_{-n+2}=0.
\]
Then there exists an object $\tilde{M}$ in $D(B,\Cscr)$ together
with an isomorphism $F(\tilde{M})\cong M$ in $D(\Cscr)_{H^\ast(B)}$.
\end{propositions}
\begin{propositions}
\label{ref-9.1.2-49}
Let $M,N\in D(B,\Cscr)$ be such that for $n\ge 2$
\[
\HH^n(H^\ast(B),\Ext^\ast_\Cscr(M,N))_{-n+1}=0.
\]
Then the map
\[
\Hom_{D(B,\Cscr)}(M,N)\r\Hom_{D(\Cscr)_{H^\ast(B)}}(FM,FN)
\]
is surjective.
\end{propositions}
\begin{propositions}
\label{ref-9.1.3-50}
Let $M,N\in D(B,\Cscr)$ be such that for $n\ge 1$
\[
\HH^n(H^\ast(B),\Ext^\ast_\Cscr(M,N))_{-n}=0.
\]
Then the map
\[
\Hom_{D(B,\Cscr)}(M,N)\r\Hom_{D(\Cscr)_{H^\ast(B)}}(FM,FN)
\]
is injective.
\end{propositions}
These results imply Proposition \ref{ref-B-1} in the introduction.
\subsection{Reminder on $A_\infty$-algebras and morphisms}
\subsubsection{$A_\infty$-algebras}
Let $A$ be a graded vector space. We denote by $\BB A=\bigoplus_{n\ge 1} (sA)^{\otimes n}$ 
the tensor coalgebra   (without counit) of $s A$. 
Sometimes we write $sa_1\otimes\cdots \otimes sa_n\in \BB A$
as a tuple $(sa_1,\ldots,sa_n)$. With this convention, the comultiplication is
given by
\[
\Delta(sa_1,\ldots,s a_n)=\sum_{i=1}^{n-1} (sa_1,\ldots,sa_i)\otimes (sa_{i+1},\ldots,sa_n).
\]
By definition, an $A_\infty$-structure on $A$ is given by a (graded) coderivation
$b:\BB A\r \BB A$ of degree~$1$ and square zero. Thus
\begin{align*}
\Delta\circ b&=(b\otimes \Id+\Id\otimes b)\circ \Delta\\
b^2&=0
\end{align*}
The coderivation $b$ is determined by its Taylor coefficients  $(b_n)_{n\ge 1}$, 
which are the compositions
\[
(s A)^{\otimes n}\xrightarrow{\text{inclusion}} \BB A\xrightarrow{b} \BB A\xrightarrow{\text{projection}} s A.
\]
The fact that $b$ is a coderivation implies 
\[
b=\sum_{p,q,r\geq 0} \Id^{\otimes p} \otimes b_q\otimes \Id^{\otimes r}.
\]
Corresponding to the $b_n$ we have
the more traditional operations 
\[
m_n:A^{\otimes n}\r A
\]
of degree $2-n$,
which are related to the $b_n$ by the formula
\[
b_n=s^{-n+1}m_n.
\]
Explicitly, in the cases $n=1$ and $n=2$
\begin{equation}
\label{ref-9.1-51}
\begin{aligned}
b_1(sa)&=-sm_1(a)\\
b_2(sa,sb)&=(-1)^{|a|}sm_2(a,b).
\end{aligned}
\end{equation}
A DG-algebra is the same as an $A_\infty$-algebra with $b_n=0$ for $n\ge 3$.
\subsubsection{$A_\infty$-morphisms}
If $A$, $C$ are $A_\infty$-algebras, an $A_\infty$-morphism $\psi:A\r C$ is by definition
a graded coalgebra morphism $\psi:\BB A\r \BB C$ commuting with the differentials. Thus
\begin{align*}
\Delta\circ \psi&=(\psi\otimes \psi)\circ \Delta\\
b_C\circ \psi&=\psi \circ b_A
\end{align*}
 Again $\psi$ is
determined by its Taylor coefficients $(\psi_n)_{n\ge 1}$, which are the compositions
\[
(s A)^{\otimes n}\xrightarrow{\text{inclusion}} \BB A\xrightarrow{\psi} \BB C\xrightarrow{\text{projection}} s C.
\]
This time we have
\[
\psi=\sum_{r,n_1,\ldots,n_r} \psi_{n_1}\otimes \psi_{n_2} \otimes \cdots \otimes \psi_{n_r}.
\]
There are no sign issues since all $\psi_n$ have degree zero. For this reason we will identify $\psi_1:sA\r sC$
with a map $\psi_1:A\r C$ (thus $\psi_1(sa)=s\psi_1(a)$).
\subsubsection{$A_\infty$-modules}
We will define $A_\infty$-modules 
over a $k$-linear Grothendieck category $\Cscr$, which we fix throughout.
If $A$ is an $A_\infty$-algebra, an $A_\infty$-$A$-module  in $C(\Cscr)$ is an object $M\in C(\Cscr)$ together
with an $A_\infty$-morphism $A\r \underline{\Hom}_\Cscr(M,M)$. Alternatively,
define
\[
\BB M=(\BB A)^+\otimes_k M,
\]
where $(-)^+$ means adjoining a counit: $(\BB A)^+=\oplus_{n\geq 0} (sA)^{\otimes n}$ and $\Delta(sa_1,\ldots,sa_n)=\sum_{i=0}^n
 (sa_1,\ldots,sa_i)\otimes (sa_{i+1},\ldots,sa_n)$ where empty parentheses are to be interpreted as 1. Then $\BB M$ is a left $\BB A$-comodule via
\[
\Delta_M (sa_1,\ldots,sa_n,m)=\sum_{i=1}^n (sa_1,\ldots,sa_i)\otimes (sa_{i+1},\ldots,sa_n,m)
\]
and an $A_\infty$-structure on $M$ is given by a $\BB A$ coderivation $b_M$
on $\BB M$ of degree one and square zero.  Thus $b_M$ satisfies 
\begin{align*}
\Delta_M\circ b_M&=(b_A\otimes \Id+\Id\otimes b_M)\circ \Delta_M\\
b_M^2&=0
\end{align*}
Needless to say, $b_M$ is again determined by its Taylor coefficients,
which are morphisms
\[
b_{M,n}:(sA)^{\otimes n-1}\otimes_k M\r M
\]
in $C(\Cscr)$.
 We have
\[
b_M=\sum_{p,q,r\geq 0}\Id_A^{\otimes p}\otimes b_{A,q}\otimes \Id_A^{\otimes r}\otimes \Id_M
+\sum_{m,n\geq 0} \Id_A^{\otimes m} \otimes b_{M,n}.
\]
We denote the category of $A_\infty$-$A$-modules in $C(\Cscr)$, with morphisms given by morphisms of complexes, by
$C_{\infty}^{\mathrm{strict}}(A,\Cscr)$. It is easy to see that this is a
Grothendieck category.
\subsubsection{$A_\infty$-morphisms between $A_\infty$-modules}
Let $A$ be an $A_\infty$-algebra and let $M,N\in C(\Cscr)$ be two
$A_\infty$-$A$-modules in $\Cscr$. An $A_\infty$-morphism $\psi:M\r
N$ is a comodule morphism of degree zero $\psi:\BB M\r \BB N$ satisfying $b_N\circ
\psi=\psi\circ b_M$. Thus $\psi$ satisfies
\[
\Delta \circ \psi =(\Id\otimes \psi)\circ \Delta.
\]

This time, the Taylor coefficients are
\[
\psi_n:(sA)^{\otimes n-1}\otimes_k M\r N
\]
and $\psi$ is given by 
\[
\psi=\sum_{p,q\geq 0}\Id^{\otimes p}\otimes \psi_q.
\]
We denote the category of $A_\infty$-$A$-modules in $C(\Cscr)$ equipped with $A_\infty$-morphisms by
$C_{\infty}(A,\Cscr)$. 

A homotopy between $A_\infty$-morphisms $\psi_1,\psi_2:M\r N$ is a comodule morphism
$h:\BB M\r \BB N$ of degree $-1$ such that $\psi_1-\psi_2=b_N\circ h+h\circ b_M$. 
\subsubsection{Units} \def\hu{\operatorname{hu}} 
 Let $A$ be an $A_\infty$-algebra. We say that $A$ has a
homological unit if $H^\ast(A)$ has a unit element $1_A$.  Let $M\in
C_\infty(A,\Cscr)$. We say that $M$ is homologically unital if $1_A$
acts as the identity on $H^\ast(M)$. All constructions for
$A_\infty$-algebras outlined above have a unital analogue in which we require that on the level of cohomology the units behave as expected.
We write $C_\infty^{\hu,\mathrm{strict}}(A, \Cscr)$, $C_{\infty}^{\hu}(A,\Cscr)$ for the corresponding
categories.
Furthermore we put
\begin{align*}
D_\infty^{\text{strict}}(A,\Cscr)&=C_\infty^{\hu,\mathrm{strict}}(A,\Cscr)[\text{Qis}^{-1}]\\
D_\infty(A,\Cscr)&=C_\infty^{\hu}(A,\Cscr)[\text{Qis}^{-1}]
\end{align*}
It follows in the usual way that homotopic maps $\psi_1,\psi_2:M\r N$ in $C_\infty(A,\Cscr)$ yield equal
maps in $D_\infty^{\text{strict}}(A,\Cscr)$ and $D_\infty(A,\Cscr)$.
\begin{lemmas} \label{ref-9.2.1-52} Let $A$ be a DG-algebra. Then the natural functors
\begin{align}
D(A,\Cscr)&\r D_\infty^{\mathrm{strict}}(A,\Cscr)\\
\label{ref-9.3-53}
D(A,\Cscr)&\r D_\infty(A,\Cscr)
\end{align}
are equivalences of categories.
\end{lemmas}
\begin{proof} The proof in the strict and non-strict cases is the same, so we consider only
\eqref{ref-9.3-53}.
If $\Cscr$ is the category of $k$-vector spaces
and we restrict ourselves to so-called  ``strictly unital''\footnote{An $A_\infty$-module is strictly unital if $b_{M,n}$, with $n\geq 2$,
vanishes as soon as one of the arguments is 1 and $b_{2,M}(\id_M\otimes 1_A)=\id_M$ where $\eta$ is the unit of $A$.} modules, this is 
\cite[Lemme 4.1.3.8]{Lefevre}. The proof in loc.\ cit.\ goes more or
less through in our setting.  The first step is the definition of
a functor
\[
A\otimes^{\infty}_A -: C_\infty(A,\Cscr)\r C(A,\Cscr).
\]
This definition is given in \cite[Lemme 4.1.1.6]{Lefevre}.  The next step is to prove that
$A\otimes^{\infty}_A -$ yields  a quasi-inverse to \eqref{ref-9.3-53} after inverting quasi-isomorphisms. This is part
of the proof of \cite[Lemme 4.1.1.6]{Lefevre}. Ultimately it reduces to (the well-known) Lemma
\ref{ref-9.2.2-54} below.
\end{proof}
\begin{lemmas} 
\label{ref-9.2.2-54}
Assume that $A$ is a homologically unital $A_\infty$-algebra
and $M\in C_\infty^{\hu}(A,\Cscr)$. Then $\BB M$ is acyclic.
\end{lemmas}
\begin{proof} 
The fact that $\BB M$ is acyclic is proved in \cite[Lemme 4.1.1.6]{Lefevre}, under
the hypothesis that $A$ and $M$ are ``strictly unital'', by providing an explicit
contracting homotopy. If we only assume that $A,M$ are homologically
unital then we cannot use this argument.

So we proceed differently. We have to show that 
$
H^\ast(\BB M)=0
$. 
We define an ascending filtration on $\BB M$
\[
F_n \BB M=\bigoplus_{m\le n} (sA)^{\otimes m} \otimes M
\]
and we consider the resulting spectral sequence.
One 
checks that the first page of this spectral sequence is
\[
\BB H^\ast(M)
\]
where we consider $H^\ast(M)$ as an $A_\infty$-module over $H^\ast(A)$ with $b_i=0$ for $i\ne 2$.
Since $H^\ast(A)$ has a true unit and $H^\ast(M)$ is truly unital, it is well-known that
$\BB H^\ast(M)$ is acyclic (for example by using
contracting homotopy given in the proof of  \cite[Lemme 4.1.1.6]{Lefevre}  alluded to above).
\end{proof}
\subsection{Proof of Proposition \ref{ref-9.1.1-48}} 
\label{ref-9.3-55}
\begin{lemmas}
\label{ref-9.3.1-56}
Let $A,C$ be two $A_\infty$-algebras over $k$ and let $\phi:A\r C$ be a $k$-linear map commuting with the differentials $m_1$ such that $H^{\ast}(\phi):H^\ast(A)\r H^\ast(C)$ is a graded algebra morphism. 
 Assume that for
all $n\ge 3$ we have
\[
\HH^{n}(H^\ast(A),H^\ast(C))_{-n+2}=0.
\]
Then there exists an $A_\infty$-morphism $\psi:A\r C$ such that $\psi_1=\phi$.
\end{lemmas}
\begin{proof}
This is close to the obstruction theory in \cite[Appendix B.4]{Lefevre} for minimal $A_\infty$-algebras. Rather
than reducing to it by invoking the fact that any $A_\infty$-algebra is $A_\infty$-homotopy equivalent to a minimal one, we give
a simple direct proof for the benefit of the reader.

We will construct $\psi$ step by step. We first put $\psi_1=\phi$. $\phi$ is compatible with the
multiplications on $A$ and $C$ up to a homotopy, which we take to be $-\psi_2$. Thus
\begin{equation}
\label{ref-9.4-57}
b_{C,2}\circ (\psi_1\otimes \psi_1) -\psi_1 \circ b_{A,2}=-b_{C,1}\circ \psi_2+\psi_2\circ (b_{A,1}\otimes \Id+\Id\otimes b_{A,1}).
\end{equation}
Assume that we have constructed $\psi_1,\ldots,\psi_n$. Let $\psi_{\le n}:\BB A\r \BB C$
be the coalgebra map such that 
\[
(\psi_{\le n})_i=
\begin{cases} 
\psi_i & i=1,\ldots,n\\
0&\text{otherwise}
\end{cases}
\]
 Assume furthermore that
\begin{equation}
\label{ref-9.5-58}
b_C\circ \psi_{\le n}=\psi_{\le n} \circ b_A\qquad \text{restricted to $(sA)^{\otimes i}$ for $1\le i\le n$} .
\end{equation}
It follows from \eqref{ref-9.4-57} that we have already achieved this for $n\le 2$.

Our aim is to construct $\psi_{n+1}$ such that $b_C\circ \psi_{\le n+1}=\psi_{\le n+1} \circ b_A$ 
when restricted to $(sA)^{\otimes i}$ for $1\le i\le n+1$. Before we start, we warn the reader that the construction
of $\psi_{n+1}$ will involve changing $\psi_n$.

 Consider
\begin{equation}
\label{ref-9.6-59}
D=b_C\circ \psi_{\le n}-\psi_{\le n} \circ b_A.
\end{equation}
Then $D:\mathbb{B}A\r \mathbb{B}C$ is a $\psi_{\le n}$ coderivation 
of degree $1$.
By construction we have $D_m=0$ for $m=1,\ldots,n$. Moreover, it is clear that we have
\begin{equation}
\label{ref-9.7-60}
b_C\circ D+D\circ b_A=0.
\end{equation}
Evaluating \eqref{ref-9.7-60} on $(sA)^{\otimes n+1}$,  we find 
\[
b_{C,1}\circ D_{n+1}+D_{n+1}\circ (\sum_{p+r=n} \Id^{\otimes p}\otimes b_{A,1}\otimes \Id^{\otimes r})=0,
\]
or succinctly
\begin{equation}
[b_1,D_{n+1}]=0.
\end{equation}
Here $[b_1,-]$ is our notation for the differential on $\Hom_k((sA)^{\otimes n+1},sC)$ induced by $b_{A,1}$ and $b_{C,1}$.

\medskip

We now evaluate \eqref{ref-9.7-60} on $(sA)^{\otimes n+2}$. We get
\begin{multline*}
b_{C,1}\circ D_{n+2}+ b_{C,2}\circ( D_{n+1}\otimes \psi_1+\psi_1\otimes D_{n+1})
\\
+D_{n+2}\circ (\sum_{p+r=n+1} \Id^{\otimes p}\otimes b_{A,1}\otimes \Id^{\otimes r})+D_{n+1}\circ (\sum_{p+r=n} \Id^{\otimes p}\otimes b_{A,2}\otimes \Id^{\otimes r})=0
\end{multline*}
Written more nicely:
\[
[b_1,D_{n+2}]+b_{C,2}\circ( D_{n+1}\otimes \psi_1)+\sum_{p+r=n} D_{n+1}\circ(\Id^{\otimes p}\otimes b_{A,2}\otimes \Id^{\otimes r})
+ b_{C,2}\circ (\psi_1\otimes D_{n+1})=0.
\]
This may be rewritten as 
\begin{equation}
\label{ref-9.9-62}
0=d_{\text{Hoch}}(\bar{D}_{n+1})
\end{equation}
where $\bar{D}_{n+1}$ is the image of $D_{n+1}$ in $H^1(\Hom_k((sA)^{\otimes n+1},sC))=\Hom_k(H^\ast(A)^{\otimes n+1},H^\ast(C))_{-n+1}$,
and where $d_{\text{Hoch}}$ represents the Hochschild differential. Thus $\bar{D}_{n+1}$ represents an element
of $\HH^{n+1}(H^\ast(A),H^\ast(C))_{-n+1}$.
\medskip

At this point we use the idea that we may modify $\psi_n$ as long as condition \eqref{ref-9.5-58}
remains valid. Let $\psi'_{\le n}$ be like $\psi_{\le n}$ except that $\psi_{n}$ is replaced by $\psi'_n=\psi_n+\delta_n$, with $\delta_n:(sA)^{\otimes n}\to sC$. Then condition \eqref{ref-9.5-58} remains valid for $\psi'_{\le n}$ provided $[b_1,\delta_n]=0$. We
will assume this.
Now let $D'$ be like $D$ but computed from $\psi'_{\le n}$. Then we find by \eqref{ref-9.6-59}
\[
D'_{n+1}=D_{n+1}+b_{C,2}\circ (\delta_n\otimes \psi_1+\psi_1\otimes \delta_n)-\delta_n\circ (\sum_{p+r=n-1} \Id^{\otimes p}\otimes b_{A,2}\otimes \Id^{\otimes r}).
\]
In other words
\[
\bar{D'}_{n+1}=\bar{D}_{n+1}\pm d_{\text{Hoch}}(\bar{\delta}_n).
\]
Combining this with  \eqref{ref-9.9-62}, together with the assumption $\HH^{n+1}(H^\ast(A),H^\ast(C))_{-n+1}=0$ (as $n\ge 2$), it
follows that we may modify $\psi_{n}$ in such a way that $\bar{D}_{n+1}=0$.

\medskip

Let $\psi_{n+1}$ be arbitrary. The condition $b_C\circ \psi_{\le n+1}=\psi_{\le n+1} \circ b_A$ when restricted
to $(sA)^{n+1}$ may be succinctly written as
\begin{equation}
\label{ref-9.10-63}
[b_1,\psi_{n+1}]=-(b_C \circ \psi_{\le n}-\psi_{\le n} \circ b_A)\mid (sA)^{\otimes n+1},
\end{equation}
which may again be rewritten as
\[
[b_1,\psi_{n+1}]=-D_{n+1}.
\]
Since $\bar{D}_{n+1}=0$, this equation has a solution.
\end{proof}
\begin{proof}[Proof of Proposition \ref{ref-9.1.1-48}]
We may assume that $M$ is a fibrant object in $C(\Cscr)$ for the standard model structure \cite{Beke}. Put $A=\underline{\Hom}_\Cscr(M,M)$.
The $H^\ast(B)$-action on $M$ is represented by a graded map $H^\ast(B)\r H^\ast(A)$. We
may lift this map to a graded linear map $\phi:B\r A$, commuting with the differentials on $B$ and $A$.

Since $H^\ast(A)=\Ext^\ast_\Cscr(M,M)$, the hypotheses together with Lemma \ref{ref-9.3.1-56}
imply that $\phi$ may be lifted to an $A_\infty$-morphism $\psi:B\r A$ such that $\psi_1=\phi$.
Then $M$ becomes a homologically unital $A_\infty$-$B$-module, that is, an object in $D_\infty(B,\Cscr)$.
The proposition now follows by invoking Lemma \ref{ref-9.2.1-52}, together with the commutative
diagram
\[
\xymatrix{
D(B,\Cscr)\ar[r]\ar[d]& D_\infty(B,\Cscr)\ar[d]\\
D(\Cscr)_{H^\ast(B)}\ar@{=}[r] & D(\Cscr)_{H^\ast(B)}
}
\]
\end{proof}
\subsection{Proof of Proposition \ref{ref-9.1.2-49}} 
Let $M,N\in D(B, \mathscr{C})$ be as in the statement of Proposition \ref{ref-9.1.2-49} and
assume that $N$ is fibrant for the model structure on $C(B,\Cscr)$ \cite[Prop.\ 5.1]{lowenvdb2}.
Then it is easy to see that $N$ is also fibrant when considered as an element of $C(\Cscr)$.
In particular,  an element $\varphi \in \text{Hom}_{D(\mathscr{C})_{H^*(B)}}(FM,FN)$ may be considered as an actual map $\varphi:M\r N$ in $C(\Cscr)$ commuting 
with the $H^\ast(B)$-action, up to homotopy.
We  will construct a morphism   $f:M\to N$ in $C_\infty(B,\Cscr)$ such that $f_1=\varphi$.   This is sufficient by Lemma \ref{ref-9.2.1-52}. 

Consider 
\[
b_{N,2}\circ (\id_B\otimes f_1)-f_1\otimes b_{M,2}.
\]
This is a map $B\otimes_k M\r N$, which we may consider as a map $B\r \underline{\Hom}_{\Cscr}(M,N)$. The latter
is zero on cohomology.  Hence on the level of complexes of $k$-vector spaces it is zero up to homotopy. Call this homotopy $-f_2:B\r \underline{\Hom}_{\Cscr}(M,N)$ and view it
as a map $B\otimes_k M\r N$ in $C(\Cscr)$. Thus we have
\begin{equation}
\label{ref-9.11-64}
b_{N,2}\circ (\id_B\otimes f_1)-f_1\otimes b_{M,2}=-b_{N,1} \circ f_2 + f_2\circ (\Id_B\otimes b_{M,1}+b_{B,1}\otimes \Id_M).
\end{equation}
Assume that we have constructed $f_1,\ldots,f_n$. Define $f_{\leq n}$ as the comodule map $\BB M\r \BB N$ given by the Taylor coefficients $(f_1,\ldots, f_n,0,\ldots)$.
Assume furthermore that 
\begin{equation}
\label{ref-9.12-65}
b_N\circ f_{\le n}=f_{\le n} \circ b_M\qquad \text{restricted to $(sB)^{\otimes i}\otimes M$ for $0\le i\le n-1$} .
\end{equation}
It follows from \eqref{ref-9.11-64} that we have already achieved this for $n\le 2$.
Our aim is now to construct $f_{n+1}$ such that $b_N\circ f_{\le n+1}=f_{\le n+1} \circ b_M$ 
when restricted to $(sB)^{\otimes i}\otimes M$ for $0\le i\le n$. As in the proof of Proposition \ref{ref-9.1.1-48}, this will involve
retroactively changing $f_n$.

Define $D=b_N\circ f_{\leq n}-f_{\leq n}\circ b_M$. Then we have $D_m=0$ for $m=0, \ldots, n$. We will now show that $[b_1,D_{n+1}]=0$. To do this, notice that  
\[b_N\circ D+ D\circ b_M=0\,.\]
Evaluate this equation on $(sB)^{\otimes n}\otimes M$ and get
\begin{equation}\label{ref-9.13-66}
b_{N,1}\circ D_{n+1}+ D_{n+1}\circ \left( \sum_{p+r=n-1} \id^{\otimes p}\otimes b_{B,1}\otimes \id^{\otimes r}\otimes \id_M\right) +D_{n+1}\circ (\id_B^{\otimes n}\otimes b_{M,1})=0
\end{equation}
which is precisely the statement that $b_{N,1}\circ D_{n+1}+ D\circ b_{M,1} =0$. 

We now want to take the adjoint map $(sB)^{\otimes n}\to \underline{\Hom}_{\Cscr}(M,N)$. To do this, first define 
\begin{align*}
\textbf{b}_{N,1}: &  \underline{\Hom}_\Cscr(M,N)\to\underline{\Hom}_\Cscr(M,N),\; \textbf{b}_{N,1}(f)=b_{N,1} \circ f\\
\textbf{b}_{M,1}: & \underline{\Hom}_\Cscr(M,N)\to\underline{\Hom}_\Cscr(M,N),\; \textbf{b}_{M,1}(f)=(-1)^{|f|} f\circ b_{M,1} \\
\textbf{D}_{n+1}: & (sB)^{\otimes n}\to \underline{\Hom}_\Cscr(M,N),\; \textbf{D}_{n+1}(sa_1,\ldots,sa_n)(-)=D_{n+1}(sa_1,\ldots,sa_n,-)
\end{align*}
Later we will also use
\begin{align*}
\textbf{b}_{N,2}: &  \underline{\Hom}_\Cscr(M,N)\otimes sB\to\underline{\Hom}_\Cscr(M,N),\; (\textbf{b}_{N,2}(f,sa))(m)=(-1)^{|sa|\cdot |f|}b_{N,2}(sa,f(m)) \\
\textbf{b}_{M,2}: & \underline{\Hom}_\Cscr(M,N)\otimes sB\to\underline{\Hom}_\Cscr(M,N),\; (\textbf{b}_{M,2}(f,sa))(m)=(-1)^{|f|} (f\circ b_{M,2})(sa,m) 
\end{align*}
Now that we have all of these maps in place, let us go back to
(\ref{ref-9.13-66}) and write down the corresponding equation for the adjoint
map $(sB)^{\otimes n}\to \underline{\Hom}_\Cscr(M,N)$,
\[\textbf{b}_{N,1} \circ \textbf{D}_{n+1} +(-1)^{|D_{n+1}|}\textbf{b}_{M,1}\circ \textbf{D}_{n+1} + \textbf{D}_{n+1}\circ \left( \sum_{p+r=n-1} \id^{\otimes p}\otimes b_{B,1}\otimes \id^{\otimes r}\right)=0\]
which, remembering that $D$ has degree 1, becomes
\[\textbf{b}_{N,1} \circ \textbf{D}_{n+1}-\textbf{b}_{M,1}\circ \textbf{D}_{n+1} + \textbf{D}_{n+1}\circ \left( \sum_{p+r=n-1} \id^{\otimes p}\otimes b_{B,1}\otimes \id^{\otimes r}\right)=0\,.\]

We may consider $\textbf{D}_{n+1}$ as an element in
$\underline{\Hom}_k((sB)^{\otimes n},\underline{\Hom}_\Cscr(M,N))$. The induced differential on $\underline{\Hom}_\Cscr(M,N)$ is
$b_{{\underline{\Hom}_\Cscr}(M,N),1}=\textbf{b}_{N,1}-\textbf{b}_{M,1}$. Then
$\underline{\Hom}_k((sB)^{\otimes n},\underline{\Hom}_\Cscr(M,N))$ is a complex with differential
$[b_1,-]$. By the computation above,
$[b_1,\textbf{D}_{n+1}]=0$. Define $\bar{\textbf{D}}_{n+1}$ as the
image of $\textbf{D}_{n+1}$ in
\[
H^1(\underline{\Hom}_k((sB)^{\otimes n},\underline{\Hom}_\Cscr(M,N)))=\underline{\Hom}_k(H^*(B)^{\otimes n},\Ext^*_\Cscr(M,N))_{-n+1}.
\]
Now evaluate $b_N\circ D+ D\circ b_M=0$ on $(sB)^{\otimes n+1}\otimes M$ to get
\begin{align*}
&b_{N,1}\circ D_{n+2} +b_{N,2}\circ(\id\otimes D_{n+1}) + \\
&+D_{n+2}\circ \left( \sum_{p+r+2=n+2} \id^{\otimes p}\otimes b_{B,1}\otimes \id^{\otimes r}\otimes \id_M + \id_B^{\otimes n+1}\otimes b_{M,1}\right) \\
&+D_{n+1}\circ \left( \sum_{p+r+2=n+1} \id^{\otimes p}\otimes b_{B,2}\otimes \id^{\otimes r}\otimes \id_M + \id_B^{\otimes n}\otimes b_{M,2}\right)=0\,.
\end{align*}
Rewrite the sums as
\begin{align*}
&b_{N,1}\circ D_{n+2} +D_{n+2}\circ ( \id_B^{\otimes n+1} \otimes b_{M,1} ) \\
&+D_{n+2}\circ \left( \sum_{p+r=n} \id^{\otimes p}\otimes b_{B,1}\otimes \id^{\otimes r}\otimes \id_M \right) \\
&+b_{N,2}\circ(\id\otimes D_{n+1})+D_{n+1}\circ (\id_B^{\otimes n}\otimes b_{M,2})  \\
&+D_{n+1}\circ \left( \sum_{p+r=n-1} \id^{\otimes p}\otimes b_{B,2}\otimes \id^{\otimes r}\otimes \id_M \right)=0
\end{align*}
By adjointness this gives maps $(sB)^{\otimes n+1}\to \underline{\Hom}_\Cscr(M,N)$ such that (remember that $D$ has degree 1, so $(-1)^{|D_{n+2}|}=-1$)
\begin{align*}
&\textbf{b}_{N,1}\circ \textbf{D}_{n+2} - \textbf{b}_{M,1}\circ \textbf{D}_{n+2} \\
&+\textbf{D}_{n+2}\circ \left( \sum_{p+r=n} \id^{\otimes p}\otimes b_{B,1}\otimes \id^{\otimes r}\right) \\
&+\textbf{b}_{N,2}\circ(\id\otimes \textbf{D}_{n+1}) -\textbf{b}_{M,2}\circ (\id\otimes \textbf{D}_{n+1})   \\
&+\textbf{D}_{n+1}\circ \left( \sum_{p+r=n-1} \id^{\otimes p}\otimes b_{B,2}\otimes \id^{\otimes r} \right)=0\,;
\end{align*}
this means that $\bar{\textbf{D}}_{n+1}$ is a cocycle in $(\underline{\Hom}_k(H^*(B)^{\otimes n},\Ext^*_\Cscr(M,N),d_{\mathrm{Hoch}})_{-n+1}$, hence an element of 
\[\HH^{n}(H^*(B),\Ext^*(M,N))_{-n+1}.\]

Now let $f'_{\leq n}$ be like $f_{\leq n}$ except that $f_n$ is
replaced by $f_n+\delta_n$, where $\delta_n:(sB)^{\otimes n-1}\otimes
M\to N$ is such that $[b_1,\delta_n]=0$, and let $D'=b_N\circ f'_{\leq
  n}-f'_{\leq n}\circ b_M$. Since $b_N\circ f'_{\leq n}-f'_{\leq
  n}\circ b_M|_{sB^{\otimes i}\otimes M}=0$ for $i=0,\ldots,n-1$, we still have
$D'_i=0$ for $i=1,\dots, n$, whereas
\[D'_{n+1}=D_{n+1}+b_{N,2}\circ(\id\otimes \delta_n)-\delta_n\circ(\sum_{p+r=n-1} \id^{\otimes p}\otimes b_{B,2}\otimes \id^{\otimes r} +\id^{\otimes n-1}\otimes b_{M,2})\]
The corresponding map $\textbf{D}_{n+1}':(sB)^{\otimes n}\r \underline{\Hom}_{\Cscr}(M, N)$ is given by
\[
\textbf{D}_{n+1}'=\textbf{D}_{n+1}+\textbf{b}_{N,2}\circ(\id\otimes \boldsymbol\delta_n)-\boldsymbol\delta_n\circ (\sum_{p+r=n-1} \id^{\otimes p}\otimes b_{B,2}\otimes \id^{\otimes r}) -\textbf{b}_{M,2}\circ (\id\otimes \boldsymbol\delta_n)
\]
where $\boldsymbol\delta_{n}: (sB)^{\otimes n-1}\to
\underline{\Hom}_\Cscr(M,N),\;
\boldsymbol\delta_{n}(sa_1,\ldots,sa_{n-1})(-)=\delta_{n}(sa_1,\ldots,sa_{n-1},-)$. Hence
$\bar{\textbf{D}}'_{n+1}= \bar{\textbf{D}}_{n+1}\pm
d_{\mathrm{Hoch}}(\bar{\boldsymbol\delta}_n)$. Since we have assumed
$\HH^{n}(H^*(B),\Ext^*(M,N))_{-n+1}=0$, it means
$\bar{\textbf{D}}_{n+1}$ is a coboundary, and hence we can assume it
is actually zero after replacing it with $
\bar{\textbf{D}}_{n+1}\pm d_{\mathrm{Hoch}}(\bar{\boldsymbol\delta}_n)$.

Given a map $f_{n+1}$, the condition that $f_{\leq n+1}$ needs to
satisfy to complete the induction step is $b_N \circ f_{\leq
  n+1}=f_{\leq n+1}\circ b_M$ when restricted to $(sB)^{\otimes
  n}\otimes M$. 
 This gives
\[
[b_1,f_{n+1}]=-(b_N\circ f_{\leq n}-f_{\leq n}\circ b_M)\qquad\text{(on $(sB)^{\otimes n}\otimes M$)}
\]
which gives
\[[b_1,f_{n+1}]=-D_{n+1}\]
and since $\bar{\textbf{D}}_{n+1}=0$, and hence $\bar{D}_{n+1}=0$, this equation has a solution.\qed

\subsection{Proof of Proposition \ref{ref-9.1.3-50}}
Let $M,N\in D(B, \mathscr{C})$ be as in the statement of Proposition \ref{ref-9.1.3-50} and
assume that $N$ is fibrant for the model structure on $C(B,\Cscr)$ \cite[Prop.\ 5.1]{lowenvdb2}. Assume that $g:M\r N$ is a morphism in $C_\infty(B,\Cscr)$ which
is sent to zero in 
$\mathrm{Hom}_{D(\mathscr{C})_{H^\ast(B)}}
(FM,FN)
\subset \mathrm{Hom}_{D(\mathscr{C})}(M,N)$. 

Assume that $g_i=0$ holds for $i\leq n$. We will change $g$ by a
homotopy $h$, with $h_i=0$ for $i\neq n,n+1$ such that $g_i=0$ for $i\le n+1$. Iterating this we
find that our original $g:M\r N$ is homopic to zero.

Consider first $n=0$. Then, since $g$ is zero in $\mathrm{Hom}_{D(\mathscr{C})}(M,N)$, we have
$h_1:M\r N$ such that $g_1=b_{N,1}\circ h_1+h_1\circ b_{M,1}$. Let
$h:\BB M\r \BB N$ be the coalgebra map $(h_1,0,\cdots)$ and put
$g'=g-(b_N\circ h+h\circ b_M)$. Then $g'_1=0$.

Now assume $n\ge 1$.
When evaluating $b_N\circ g=g\circ b_M$ on $(sB)^{\otimes n}\otimes M$ we get
\[[b_1,g_{n+1}]=0\,.\]
Hence $g_{n+1}$ is a cocycle in the complex
$\underline{\Hom}_k((sB)^{\otimes n}\otimes M,N)$ with differential
$[b_1,-]$. We may consider it as an element of
$\underline{\Hom}_k((sB)^{\otimes n},\underline{\Hom}_\Cscr(M,N))_{0}$ by adjointness. Call the
adjoint map $\textbf{g}_{n+1}$, and $\bar{\textbf{g}}_{n+1}$ its image
in the cohomology $H^0(\underline{\Hom}_k((sB)^{\otimes n}\otimes
M,N))=\underline{\Hom}_k(H^*(B)^{\otimes n},\Ext^*_\Cscr(M,N))_{-n}$.

Now evaluate $b_N\circ g=g\circ b_M$ on $(sB)^{\otimes n+1}\otimes M$ to get:
\begin{align*}
&b_{N,1}\circ g_{n+2} +b_{N,2}\circ(\id\otimes g_{n+1}) + \\
&-g_{n+2}\circ \left( \sum_{p+r+2=n+2} \id^{\otimes p}\otimes b_{B,1}\otimes \id^{\otimes r}\otimes \id_M + \id_B^{\otimes n+1}\otimes b_{M,1}\right) \\
&-g_{n+1}\circ \left( \sum_{p+r+2=n+1} \id^{\otimes p}\otimes b_{B,2}\otimes \id^{\otimes r}\otimes \id_M + \id_B^{\otimes n}\otimes b_{M,2}\right)=0
\end{align*}
By adjointness this gives maps $(sB)^{\otimes n+1}\to \underline{\Hom}_\Cscr(M,N)$ such that
\begin{align*}
&\textbf{b}_{N,1}\circ \textbf{g}_{n+2} - \textbf{b}_{M,1}\circ \textbf{g}_{n+2} \\
&-\textbf{g}_{n+2}\circ \left( \sum_{p+r=n} \id^{\otimes p}\otimes b_{B,1}\otimes \id^{\otimes r}\right) \\
&+\textbf{b}_{N,2}\circ(\id\otimes \textbf{g}_{n+1}) -\textbf{b}_{M,2}\circ (\id\otimes \textbf{g}_{n+1})   \\
&-\textbf{g}_{n+1}\circ \left( \sum_{p+r=n-1} \id^{\otimes p}\otimes b_{B,2}\otimes \id^{\otimes r} \right)=0
\end{align*}
Since $g$ has degree zero, this means that
$d_{\mathrm{Hoch}}(\bar{\textbf{g}}_{n+1})=0$, i.e. that
$\bar{\textbf{g}}_{n+1}$ is a cocycle in $(\underline{\Hom}_k(H^*(B)^{\otimes
  n},\Ext^*_\Cscr(M,N),d_{\mathrm{Hoch}})_{-n}$ hence an element of
\[\HH^{n}(H^*(B),\Ext^*_\Cscr(M,N))_{-n}\,;\]
which we have assumed to be zero for $n\geq 1$, hence $\bar{\textbf{g}}_{n+1}$ is a coboundary. Which means that there exists a $\bar{\textbf{h}}_n \in (\underline{\Hom}_k(H^*(B)^{\otimes n-1},\Ext^*_\Cscr(M,N))_{-n}$ such that $\bar{\textbf{g}}_{n+1}=d_{\mathrm{Hoch}}(\bar{\textbf{h}}_n)$. 

We may lift $\bar{\textbf{h}}_n$ to a map $\textbf{h}_n:(sB)^{\otimes n-1}\to \Hom_\Cscr(M,N)$ such that $[b_1,\textbf{h}_n]=0$, or equivalently by adjointness a map $h_n:(sB)^{\otimes n-1}\otimes M\to N$. Because $\bar{\textbf{g}}_{n+1}=d_{\mathrm{Hoch}}(\bar{\textbf{h}}_n)$, their difference is a boundary:
\begin{equation}\label{ref-9.14-67}
  \textbf{g}_{n+1}-(\textbf{b}_{N,2}(\id\otimes \textbf{h}_n)+ \textbf{b}_{M,2}\circ(\id\otimes \textbf{h}_n) -(-1)^{|\textbf{h}|=-1} \textbf{h}_n\circ(\sum \id^{\otimes p}\otimes b_{B,2}\otimes \id^{\otimes r}))=[b_1,\textbf{h}_{n+1}]
\end{equation}
for some $h_{n+1}$ in $\underline{\Hom}_k((sB)^{\otimes n},\underline{\Hom}_\Cscr(M,N))_{-1}$. 

Let $h$ be the comodule map $\BB M\r \BB N$ such that $h_i=0$ for $i\neq n,n+1$ and $h_n$, $h_{n+1}$ as above.
 These $h$  will be the required homotopy. In fact, by rewriting (\ref{ref-9.14-67}), we obtain  
\begin{align*}
\textbf{g}_{n+1}=&\textbf{b}_{N,1}\circ \textbf{h}_{n+1} - \textbf{b}_{M,1}\circ \textbf{h}_{n+1} \\
&+\textbf{h}_{n+1}\circ \left( \sum_{p+r=n} \id^{\otimes p}\otimes b_{B,1}\otimes \id^{\otimes r}\right) \\
&+\textbf{b}_{N,2}\circ(\id\otimes \textbf{h}_{n}) -\textbf{b}_{M,2}\circ (\id\otimes \textbf{h}_{n})   \\
&+\textbf{h}_{n}\circ \left( \sum_{p+r=n-1} \id^{\otimes p}\otimes b_{B,2}\otimes \id^{\otimes r} \right)
\end{align*}
which, by adjointness, gives us ($h$ has degree -1 so this accounts for the change in signs)
\begin{align}\label{ref-9.15-68}
g_{n+1}=&b_{N,1}\circ h_{n+1} + b_{M,1}\circ h_{n+1} \\
&+h_{n+1}\circ \left( \sum_{p+r=n} \id^{\otimes p}\otimes b_{B,1}\otimes \id^{\otimes r}\right) \notag\\
&+b_{N,2}\circ(\id\otimes h_{n}) +b_{M,2}\circ (\id\otimes h_{n})   \notag\\
&+h_{n}\circ \left( \sum_{p+r=n-1} \id^{\otimes p}\otimes b_{B,2}\otimes \id^{\otimes r} \right) \notag
\end{align}
Define $g'=g-(b_N\circ h+h\circ b_M)$. It follows from \eqref{ref-9.15-68} that $g'_{n+1}=0$ and hence we are done. \qed
\appendix
\section{Proof of Proposition \ref{ref-6.1-34} for the two loop quiver}
\label{ref-A-69}
Let $Q$ be the two loop quiver, and let $\alpha\in \NN-\{0\}$. The proof of Proposition
\ref{ref-6.1-34} depends crucially on the construction of a representation $U$ over
$LQ$ for $L/k$ a field of transcendence degree three such that $\HH^3(L,\End_{kQ}(U))\neq 0$.

Assume we try to find our $U$ as defined over the generic point of a closed
subvariety of dimension three in a suitable $M_{\alpha,\lambda,Q}$.
In this case there is only one possibility
for $\lambda$, namely $\lambda=0$. If $\alpha=1$ then $\dim M_{\alpha,0,Q}=2$  which it too small
for our purposes. However, when $\alpha>1$ then $\alpha$ is divisible and so Proposition
\ref{ref-2.5.2-11} does not apply.

We proceed as follows. Assume $n=\alpha>1$ and put $A=kQ$. Then it is well known \cite{Procesi2} that $U_{n,A}$ is nonempty
and that $\Ascr_n$ is not split. In fact, it is generically a division algebra\footnote{Suitably called
the generic division algebra of index $n$.}. Let $x$ be the generic point of
a three dimensional irreducible subvariety of $U_{n,A}$ and put $K=k(x)$. Let $C=i^\ast_x\Ascr_n$. Then
$C$ is a central simple algebra of rank $n^2$ over $K$. Thus we have $C=M_m(D)$ 
 where $D$ is a division algebra such that $[D:K]=p^2$ with $n=pm$. Let $L/K$ be a maximal subfield
of $D$. Then $L\otimes_K C=M_m(L\otimes_K D)=M_{mp}(L)=\End_L(U)$ where $U=L^{n}$.  As in the proof of
Lemma \ref{ref-2.2.1-3} we obtain 
\[
\End_A(U)=\End_C(U)
\]
The following lemma
does what we want by the Hochschild-Kostant-Rosenberg theorem \cite{HKR}. 
\begin{lemma} One has
$
\HH^\ast(L,\End_C(U))\cong \HH^\ast(K,L)
$.
\end{lemma}
\begin{proof}
By Morita theory we have 
\[
\End_C(U)=\End_D(U_0)
\]
where $U_0=L^p$. So we may and we will assume that $C=D$, $U=U_0$, $m=1$, $n=p$.

\medskip

Since $D/K$ is central simple we have an isomorphism of algebras
\[
D\otimes_K D^\circ\r\End_K(D):d\otimes d'\mapsto (x\mapsto dxd')
\]
Taking centralisers of $1\otimes_K L$ and $L\otimes_KL$ on both sides, we find corresponding
isomorphisms
\begin{equation}
\label{ref-A.1-70}
D\otimes_K L\cong \End_L(D_L)
\end{equation}
\begin{equation}
\label{ref-A.2-71}
L\otimes_K L\cong \End_{L\otimes_K L}({}_L D_L),
\end{equation}
where $D_L$, ${}_L D_L$ denote $D$ viewed respectively as a right $L$-module and as a left-right $L$-bimodule. Since $L\otimes_K L$ is a direct sum of fields, \eqref{ref-A.2-71} implies that ${}_L D_L$ is isomorphic
to $L\otimes_K L$ as $L$-bimodules.  Since we also have  $L=\Hom_K(L,K)$ as $L$-vector spaces (both are one-dimensional),
we obtain that ${}_L D_L$ is also isomorphic to
\begin{equation}
\label{ref-A.3-72}
{}_LD_L \cong \Hom_K(L,K)\otimes_K L=\Hom_K(L,L)
\end{equation}
as $L$-bimodules.

The equality \eqref{ref-A.1-70} implies  that $U\cong D_L$ with the standard left $D$-module structure on $D$. Thus
\[
\End_D(U)=\End_D(D_L)=D^\circ
\]
with the $L$-bimodule structure on $D^\circ$ given by the left and right action. Combining this with \eqref{ref-A.3-72}
we get
\begin{align*}
\HH^\ast(L,\End_C(U))
&\cong \HH^\ast(L,\Hom_K(L,L))\\
&\cong \HH^\ast(K,L)
\end{align*}
where in the last line we have used Corollary \ref{ref-4.5-25}.
\end{proof}
\def\cprime{$'$} \def\cprime{$'$} \def\cprime{$'$}
\providecommand{\bysame}{\leavevmode\hbox to3em{\hrulefill}\thinspace}
\providecommand{\MR}{\relax\ifhmode\unskip\space\fi MR }
\providecommand{\MRhref}[2]{%
  \href{http://www.ams.org/mathscinet-getitem?mr=#1}{#2}
}
\providecommand{\href}[2]{#2}

\end{document}